\documentclass[11pt]{amsproc}

\usepackage{mathtext}
\usepackage[cp1251]{inputenc}

\usepackage{bm}

\usepackage[dvips]{graphicx}
\usepackage{amsmath}
\usepackage{amssymb}
\usepackage{amsxtra}

\usepackage{caption}
\def\N{{{\Bbb N}}}
\def\Z{{{\Bbb Z}}}
\def\T{{{\Bbb T}}}
\def\R{{\Bbb R}}
\def\C{{\Bbb C}}
\def\l{{\lambda }}
\def\a{{\alpha }}
\def\D{{\Delta }}

\def\a{{\alpha}}

\def\d{{\delta}}

\def\s{{\sigma}}
\def\vp{{\varphi}}
\def\t{{\theta }}

\def\u{{\bm u}}

\def\j{{\bm j}}
\def\k{{\bm k}}

\def\x{{\bm x}}

\def\rr{{\bm r}}
\def\elll{{\bm \ell}}

\newcommand{\h}{\widehat}
\newcommand{\w}{\widetilde}

\def\){\right)}
\def\({\left(}
\def\supp{\operatorname{supp}}

\def\sign{\operatorname{sign}}

\numberwithin{equation}{section}

\newtheorem{corollary}{Corollary}[section]
\newtheorem{lemma}{Lemma}[section]
\newtheorem{theorem}{Theorem}[section]
\newtheorem{proposition}{Proposition}[section]
\newtheorem{remark}{Remark}[section]
\newtheorem{definition}{Definition}[section]

\par

\begin{document}

\title[Approximation by quasi-interpolation operators and Smolyak's algorithm]{Approximation by quasi-interpolation operators and Smolyak's algorithm}

\author[Yurii
Kolomoitsev]{Yurii
Kolomoitsev$^{\text{a, b, 1}}$}
\address{Universit\"at zu L\"ubeck,
Institut f\"ur Mathematik,
Ratzeburger Allee 160,
23562 L\"ubeck}

\email{kolomoitsev@math.uni-luebeck.de}

\thanks{$^\text{a}$Universit\"at zu L\"ubeck,
Institut f\"ur Mathematik,
Ratzeburger Allee 160,
23562 L\"ubeck, Germany}

\thanks{$^\text{b}$Institute of Applied Mathematics and Mechanics of NAS of Ukraine, General Batyuk Str.~19, Slov’yans’k, Donetsk region, Ukraine, 84116}

\thanks{$^1$Supported by DFG project KO 5804/1-2}


\thanks{E-mail address: kolomoitsev@math.uni-luebeck.de}

\date{\today}
\subjclass[2010]{41A25, 41A63, 42A10, 42A15, 41A58, 	41A17, 	42B25, 42B35} \keywords{Smolyak algorithm, quasi-interpolation operators,  Kantorovich operators, Besov--Tribel--Lizorkin spaces of mixed smoothness, error estimates, Littlewood--Paley-type characterizations}

\begin{abstract}
We study approximation of multivariate periodic functions from Besov and Triebel--Lizorkin spaces of dominating mixed smoothness by the Smolyak algorithm constructed using a special class of quasi-interpolation operators of Kantorovich-type. These operators are defined similar to the classical sampling operators by replacing samples  with the average values of a function on small intervals (or more generally with sampled values of a convolution of a given function with an appropriate kernel). In this paper, we estimate the rate of convergence of the corresponding Smolyak algorithm in the $L_q$-norm for functions from the Besov spaces $\mathbf{B}_{p,\theta}^s(\mathbb{T}^d)$ and the Triebel--Lizorkin spaces $\mathbf{F}_{p,\theta}^s(\mathbb{T}^d)$ for all $s>0$ and admissible $1\le p,\theta\le \infty$ as well as provide analogues of the Littlewood--Paley-type characterizations of these spaces in terms of families of quasi-interpolation operators.
 \end{abstract}

\maketitle

\section{Introduction}

In this paper, we study approximation of multivariate periodic functions from Besov and Triebel--Lizorkin spaces of dominating mixed smoothness  by the Smolyak algorithm constructed using general quasi-interpolation operators.  Recall
that for a given family of univariate operators $Y=(Y_j)_{j\in \Z_+}$, the Smolyak algorithm is defined as follows:
\begin{equation}\label{sm1}
  T_n^Y:=\sum_{\j\in \Z_+^d:\,|\j|_1\le n} \Delta_\j^Y,\quad \Delta_\j^Y:=\prod_{i=1}^d (Y_{j_i}^i-Y_{j_i-1}^i),
\end{equation}
where $Y_j^i$ denotes the univariate operator $Y_j$ acting on functions in the variable $x_i$ and $Y_{-1}=0$.

There are many works dedicated to the study of approximation properties of the Smolyak algorithm applied to different families $Y$ (see, e.g.,~\cite[Chapters~4 and~5]{DTU18} and references therein).
The standard families $Y$, for which the Smolyak algorithm has been well studied, are the sampling operators $I=(I_j)_{j\in \Z_+}$ and the de la Vall\'ee Poussin  means $V=(V_j)_{j\in \Z_+}$ defined by
\begin{equation}\label{I}
  I_j(f)(x)=2^{-j}\sum_{k\in A_j} f(x_k^j)v_j(x-x_k^j), \quad x_k^j=\frac{\pi k}{2^{j-1}},
\end{equation}
and
\begin{equation}\label{S}
  V_j(f)(x)=\frac{1}{2\pi}\int_{\T}f(t)v_j(x-t)dt,
\end{equation}
respectively. Here,  $A_j=[-2^{j-1},2^{j-1})\cap \Z$ and  $v_j(x)=\sum_{k} \eta(2^{-j}k)e^{{\rm i}kx}$ is a de la Vall\'ee Poussin kernel with some generator $\eta$ (as a rule $\eta$ is a continuous function, $\supp \eta\subset (-1,1)$, and $\eta(\xi)=1$ if $|\xi|\le \rho$ for some $\rho>0$).

Let us recall some $L_p$-error estimates for the Smolyak algorithms $T_n^I$ and $T_n^V$ in the case of
functions from the Besov space of dominating mixed smoothness $\mathbf{B}_{p,\t}^s(\T^d)$ with $1\le p,\t\le \infty$ (see Definition~\ref{d3}). We have (see, e.g.,~\cite{SU07} and~\cite[Chapters~4 and~5]{DTU18}):

(i) if $s>0$, then
\begin{equation}\label{S2}
   \sup_{f\in U\mathbf{B}_{p,\t}^s }\|f-T_n^V(f)\|_{L_p(\T^d)} \asymp
\left\{
    \begin{array}{ll}
     2^{-s n}n^{(d-1)(\frac1p-\frac1\t)}, & \hbox{$p\le 2$ and $p\le \t$,} \\
     2^{-s n}n^{(d-1)(\frac12-\frac1\t)}, & \hbox{$2<p$ and $2<\t$,} \\
     2^{-s n}, & \hbox{otherwise,}
    \end{array}
\right.
\end{equation}

(ii) if $s>1/p$, then
\begin{equation}\label{I2}
   \sup_{f\in U\mathbf{B}_{p,\t}^s }\|f-T_n^I(f)\|_{L_p(\T^d)} \asymp 2^{-s n}n^{(d-1)(1-\frac1\t)},
\end{equation}
where $U\mathbf{B}_{p,\t}^s$ denotes the unit ball in  $\mathbf{B}_{p,\t}^s(\T^d)$. In particular, estimates~\eqref{S2} and~\eqref{I2} show that the Smolyak algorithm generated by convolution-type operators has better approximation properties in $L_p(\T^d)$ spaces than the corresponding algorithm generated by sampling operators.
In the case of Triebel--Lizorkin spaces of dominating mixed smoothness, similar error estimates can be found, e.g., in~\cite{BU17, DTU18}.

The main goal of this paper is to obtain analogues of estimates~\eqref{S2} and~\eqref{I2} for the Smolyak algorithm applied to a family of general quasi-interpolation operators $Q=(Q_j)_{j\in \Z_+}$ given by 
\begin{equation}\label{Q0}
  Q_j(f)(x)=Q_j(f,\vp,\w\vp)(x)=2^{-j}\sum_{k\in A_j} (f*\w\vp_j)(x_k^j)\vp_j(x-x_k^j),
\end{equation}
where $\vp=(\vp_j)_{j\in \Z_+}$ is a family of univariate trigonometric polynomials, $\w\vp=(\w\vp_j)_{j\in \Z_+}$ is a family of functions/distributions, and the convolution $f*\w\vp_j$ is defined in some suitable way.
A classical example of operator~\eqref{Q0} is the mentioned above sampling operator $I_j$
(in this case, $\w\vp_j=\delta$ is the periodic delta-function and $\vp_j=V_j$ is the de la Vall\'ee Poussin kernel).
Other related operators can be obtaining from~\eqref{Q0} if, for example, $\w\vp_j$ is a linear combination of the shifted delta functions and $\vp_j$ is some classical kernel (Dirichlet, Riesz, Fej\'er, etc.).
One has a quite different class of operators if $\w\vp_j$ is a certain integrable function, for example if $\w\vp_j$ is a characteristic function of a bounded interval in~$\T$.
In this case, the operators of form~\eqref{Q0} are usually called Kantorovich-type operators.
In the recent years, such operators have been intensively studied in many works, see, e.g., \cite{BBSV, CV19, CSV20, KKS18, KS17, KS20, KS19} for the non-periodic case and~\cite{JBU02, KKS20, KP20} for recent results in the periodic case.
It is worth noting that the Kantorovich-type operators have several advantages over the interpolation and sampling operators. 
Particularly, using the averages
of a function instead of the sampled values $f(x_k^j)$ allows to deal with discontinues signals and to reduce the so-called time-jitter errors, which is very useful in digital image processing.
Also, the general quasi-interpolation operators~\eqref{Q0} can be used in applied problems, which contain noisy data and the functional information is provided by other means than point evaluation (for example integrals, averaging, divided differences etc.), see, e.g.,~\cite{CSV20}.

Here, we mention some approximation properties of the periodic Kantorovich-type operators.
Let $\vp_j=v_j$ and let $\w\vp_j=2^{j+\s}\chi_{[-\pi2^{-j-\s},\pi2^{-j-\s}]}$ be a normalized characteristic function of the interval $[-\pi2^{-j-\s},\pi2^{-j-\s}]$, where $\s\ge 2$. Denote $K_j=Q_j(\cdot,\vp,\w\vp)$, i.e.,  
\begin{equation}\label{K1}
   K_j(f)(x)=\sum_{k\in A_j} \frac{2^{\s-1}}{\pi}\int_{-\pi 2^{-j-\s}}^{\pi2^{-j-\s}} f\big(t+x_k^j\big){\rm d}t\, v_j\big(x-x_k^j\big).
\end{equation}
It was proved in~\cite{KP20} that for every $f\in L_p(\T)$, $1\le p\le \infty$, and $j\in \Z_+$, we have
\begin{equation}\label{K2}
   \|f-K_j(f)\|_{L_p(\T)}\asymp \sup_{|h|\le 2^{-j}}\|\D_h^2 f\|_{L_p(\T)}, 
\end{equation}
where $\D_h^2 f(x)=f(x+h)-2f(x)+f(x-h)$ and $\asymp$ is a two-sided inequality with positive constants depending only on $p$, $\s$, and $\rho$.
At the same time,  if in~\eqref{K1}  we replace the de la Vall\'ee Poussin kernel $v_j$ by its modification \ $v_j^*(x)=\sum_{\ell\neq 0}\frac{\pi2^{-j-\s}\ell}{\sin \pi2^{-j-\s}\ell}\eta(2^{-j}\ell) {e}^{{\rm i} \ell x}$, then for the corresponding Kantorovich-type operator $K_j^*$, we have (see~\cite{KP20}) the following estimate:
\begin{equation}\label{K3}
   \|f-K_j^*(f)\|_{L_p(\T)}\lesssim E_{\rho 2^j}(f)_p,\quad f\in L_p(\T),
\end{equation}
where  $E_{\rho 2^j}(f)_p$ denotes the $L_p$-error of the best approximation of $f$ by trigonometric polynomials of degree at most $\rho 2^j$.
In this paper, we show (see Theorems~\ref{tb3} and~\ref{tf3} as well as Examples) that although the Kantorovich-type operator $K_j^*$ is well-defined for all $f\in L_1(\T)$ and has almost the same approximation properties as the da la Vall\'ee Poussin means $V_j$ (recall that $\|f-V_j(f)\|_p\lesssim E_{\rho 2^j}(f)_p$, see, e.g.,~\cite[p.~137]{TB}), the $L_p$-approximation order of $T_n^{K^*}$  coincides with the corresponding order in the case of approximation by the sampling operator $T_n^I$ (cf.~\eqref{I2}). However, in contrast with $T_n^I$, the analogue of estimate~\eqref{I2} for $T_n^{K^*}$ is valid for all $s>0$. More generally, considering sufficiently wide classes of quasi-interpolation operators~\eqref{Q0},  we obtain analogues of estimate~\eqref{I2} in the $L_q$-norm for functions from the Besov spaces $\mathbf{B}_{p,\t}^s(\T^d)$ and the Triebel--Lizorkin spaces $\mathbf{F}_{p,\t}^s(\T^d)$ for all $s>0$ and admissible $1\le p,\t\le \infty$.

In the study of approximation (function recovery) by the Smolyak algorithms, 
Littlewood--Paley-type characterizations of Sobolev and Besov spaces are among the most important tools. 
Recall the following well-known relations:
\begin{equation}\label{LP1}
  \bigg\Vert \bigg(\sum_{j\in \Z_+} 2^{2sj}|\D_j(f)|^2\bigg)^{1/2} \bigg\Vert_{L_p(\T)} \asymp \Vert f\Vert_{W_p^s (\T)},\quad s\ge 0,\, 1<p<\infty,
\end{equation}
and
\begin{equation}\label{LP2}
\(\sum_{j\in \Z_+} 2^{s\t j} \|\D_j(f)\|_{L_p(\T)}^\t\)^{1/\t} \asymp \|f\|_{B_{p,\t}^s(\T)},\quad s>0,\, 1<p<\infty, \, 1\le \t\le\infty,
\end{equation}
where $\D_j (f)(x) =\sum_{2^{j}-1\le |\ell|< 2^j} \h f(\ell) {\rm e}^{{\rm i} \ell x}$, and $W_p^s(\T)$ and $B_{p,\t}^s(\T)$ are the univariate Sobolev and Besov spaces, respectively.
The corresponding inequalities (characterizations) in the multi-dimensional case
have also been well-known for a long time (see, e.g.,~\cite{ST87} and~\cite{N75}).
We are interested in similar relations but with other building blocks, particularly with a certain "discrete" convolution instead of the trigonometric polynomials $\D_j$.   
As a rule, such relations are called the discrete Littlewood--Paley-type characterizations of smooth function spaces.
The first results of this type for the classical isotropic Besov and Sobolev spaces were obtained by Sickel~\cite{S91}  (see also~\cite{S92} for the non-periodic case).  In the case of periodic Besov spaces of dominating mixed smoothness, the corresponding characterizations were derived by Dihn D\~ung~\cite{D01}, the case of Sobolev spaces was recently considered by Byrenheid and Ullrich in~\cite{BU17} (see also~\cite{BDSU16}). Note that in the mentioned papers, the  corresponding analogues of~\eqref{LP1} and~\eqref{LP2} were obtained by means of the classical sampling and interpolation operators of type~\eqref{I}.
As usual in such kind of problems, the results were achieved under the following natural restrictions on the parameter of smoothness: $s>\max\{1/p,1/2\}$ (the case of Sobolev spaces) and $s>1/p$ (the case of Besov spaces).

In this paper, we obtain analogues of the "discrete" Littlewood--Paley-type characterizations for the Besov spaces $\mathbf{B}_{p,\t}^s(\T^d)$ and the Triebel--Lizorkin spaces $\mathbf{F}_{p,\t}^s(\T^d)$ for all $s>0$ using the general quasi-interpolation operators~\eqref{Q0} for the construction of the corresponding building blocks.
Our method is essentially based on the approaches developed in~\cite{BU17} and~\cite{D01}.

The paper is organized as follows.  In Section~\ref{sec2} we introduce some basic notation used throughout the paper.
Section~\ref{sec3} is dedicated to definitions and auxiliary results. 
In Section~\ref{sec4} we present and prove our main results for functions belonging to the Besov spaces of dominating mixed smoothness. Particularly, in Theorem~\ref{tb1} we consider Littlewood--Paley-type characterizations in terms of a family of general quasi-interpolation operators $Q=(Q_j)_{j\in \Z_+}$ and in Theorems~\ref{tb2} and~\ref{tb3}, we provide estimates for approximation by the Smolyak algorithm $T_n^Q$ in $L_q$-norm for functions from the Besov classes $\mathbf{B}_{p,\t}^s(\T^d)$ .  In Section~\ref{sec5} we present similar results in the case of the Triebel--Lizorkin spaces $\mathbf{F}_{p,\t}^s(\T^d)$. Section~\ref{sec6} is dedicated to examples. Finally, in Section~\ref{sec7} we give some remarks about approximation properties of the Smolyak algorithm applied to general sampling operators.

\section{Notation}\label{sec2}

We use the standard multi-index notations.
    Let $\N$ be the set of positive integers, $\R^d$ be the $d$-dimensional Euclidean space,
    $\Z^d$ be the integer lattice  in $\R^d$, $\Z_+^d:=\{\x\in\Z^d:~x_i\geq~{0}, i=1,\dots, d\}$,
    $\T^d=\R^d\slash2\pi\Z^d$ be the $d$-dimensional torus. 
    Further, let  $\x = (x_1,\dots, x_d)$ and
    $\k =(k_1,\dots, k_d)$ be vectors in $\R^d$.
    Then $(\x, \k):=x_1k_1+\dots+x_dk_d$ and for a given  $i\in \{1,\dots,d\}$, we denote
    $\x_i^d=(x_i, x_{i+1},\dots,x_d)\in \R^{d-i+1}$.
		If ${\bm j}\in\Z^d_+$,  we set
    $|\bm{j}|_1=\sum_{k=1}^d j_k$ and $2^{\j}=(2^{j_1},\dots,2^{j_d})$.

By $L_p(\Omega)$, $1\le p\le \infty$, we denote the space of all measurable
functions $f : \Omega\mapsto \C$ such that
the following norms are finite
$$
\|f\|_{L_p(\Omega)} = \Big(\int_{\Omega}|f(\x)|^p {\rm d}\x\Big)^{1/p}\quad \text{for}\quad  1\le p<\infty
$$
and
$$
\|f\|_{L_\infty(\Omega)}={{\rm ess} \sup}_{\x\in \Omega} |f(\x)| \quad \text{for}\quad  p=\infty.
$$
As usual, we take $L_\infty=C(\Omega)$ (the set of all bounded and continuous functions $\Omega$) and write $\|\cdot\|_p=\|\cdot\|_{L_p(\T^d)}$.  
The sequence space $\ell_p^r$, where $p\ge 1$ and $r\in \R$, is defined by
$$
\ell_p^r=\Big\{a=(a_\j)_{\j\in \Z_+^d}\subset \C\,:\, \|a\|_{\ell_p^r}=\Big(\sum_{\j\in \Z_+^d} 2^{r|\j|_1 p}|a_\j|^p\Big)^{1/p}<\infty\Big\}.
$$
In the case $r=0$, we denote $\ell_p=\ell_p^0$.
Next, for any $1\le p\le \infty$, we define the spaces $L_p(\T^d,\ell_\t)$ and $\ell_\t(\T^d,L_p)$ as a collection of all sequences of functions $(f_\j)_{\j\in \Z_+^d}\subset L_p(\T^d)$ with finite norms
$$
\|f_\j\|_{L_p(\T^d,\ell_\t)}=\left\{
                          \begin{array}{ll}
                            \Big\|\(\sum_{\j\in \Z_+^d}|f_\j|^\t\)^{1/\t}\Big\|_p, & \hbox{$1\le \t<\infty$,} \\
                            \|\sup_{\j\in \Z_+^d}|f_\j|\|_p, & \hbox{$\t=\infty$,}
                          \end{array}
                        \right.
$$
and
$$
\|f_\j\|_{\ell_\t(\T^d,L_p)}=\left\{
                          \begin{array}{ll}
                            \(\sum_{\j\in \Z_+^d}\|f_\j\|_p^\t\)^{1/\t}, & \hbox{$1\le \t<\infty$,} \\
                            \sup_{\j\in \Z_+^d}\|f_\j\|_p, & \hbox{$\t=\infty$,}
                          \end{array}
                        \right.
$$
respectively.
The unit ball in some normed vector space $X$ is denoted by $UX$.

If $f\in L_1(\T^d)$, then
$$
\h f(\k)=(2\pi)^{-d}\int_{\T^d} f(\x){e}^{-{\rm i}(\k,\x)}{\rm d}\x,\quad \k\in\Z^d,
$$
denotes the $k$-th Fourier coefficient of $f$.
As usual, the convolution of integrable functions $f$ and $g$ is given by
$$
(f*g)(\x)=(2\pi)^{-d}\int_{\T^d} f(\x-{\bm  t})g({\bm  t}){\rm d}{\bm  t}.
$$
By $\mathcal{T}_\j^d$, $j\in \Z_+^d$, we denote the following set of trigonometric polynomials:
$$
\mathcal{T}_\j^d={\rm span}\left\{e^{{\rm i}(\k,\x)}\,:\,\k\in  [-2^{j_1},2^{j_1})\times\dots\times [-2^{j_d},2^{j_d})\cap \Z^d\right\}.
$$
We say that a sequence of trigonometric polynomials $(t_\j)_{\j\in \Z_+^d}$ belongs to the space $L_p^\mathcal{T}(\T^d,\ell_\t)$ if
$t_\j\in \mathcal{T}_\j^d$ for any $\j\in \Z_+^d$ and $\|(t_\j)_{\j}\|_{L_p(\T^d,\ell_\t)}<\infty$.

Throughout the paper, we use the notation
$
\, A \lesssim B,
$
with $A,B\ge 0$, for the estimate
$\, A \le C\, B,$ where $\, C$ is a positive constant independent of
the essential variables in $\, A$ and $\, B$ (usually, $f$, $j$, and $n$). 
If $\, A \lesssim B$
and $\, B \lesssim A$ simultaneously, we write $\, A \asymp B$ and say that $\, A$
is equivalent to $\, B$.
For two function spaces
$\, X$ and $\, Y,$ we will use the notation
$
\, Y \hookrightarrow X
$
if
$\, Y \subset X$ and $\, \| f\|_X \lesssim \| f\|_Y$ for all $\, f \in
Y.$
For $1\le p\le\infty$, $p'$ is given by
$\frac{1}{p}+\frac{1}{p'}=1.$

\section{Basic definitions and auxiliary results}\label{sec3}

\subsection{Smooth function spaces}\label{sec3.1}

Let us give the classical analytical definition of periodic Besov spaces $\mathbf{B}_{p,\t}^{\rr}(\T^d)$ and Triebel--Lizorkin spaces $\mathbf{F}_{p,\t}^{\rr}(\T^d)$ with dominating mixed smoothness. First, we introduce a smooth resolution of unity.
We say that a system $\phi=(\phi_j)_{j\in \Z_+}\subset C_0^\infty(\R)$ belongs to the class $\Phi(\R)$ if it satisfies the following conditions:
\begin{itemize}
  \item[$(i)$] $\supp \phi_0\subset \{\xi\,:\,|\xi|\le 2\}$;
  \item[$(ii)$] $\supp \phi_j\subset \{\xi\,:\,2^{j-1}\le |\xi|\le 2^{j+1}\}$, $j=1,2,\dots$;
  \item[$(iii)$] for all $\ell\in \Z_+$, it holds $\sup_{\xi\in \R,\, j\in \Z_+}2^{j\ell}|\phi_j^{(\ell)}(\xi)|\le c_\ell<\infty$;
  \item[$(iv)$] $\sum_{j=0}^\infty \phi_j(\xi)=1$ for all $\xi\in \R$.
\end{itemize}
Next, for any $f\in L_1(\T^d)$ and $\j\in \Z_+^d$, we denote
$$
\d_\j(f)(\x)=\sum_{\k\in \Z^d}\phi_{j_1}(k_1)\dots \phi_{j_d}(k_d)\h f(\k) {e}^{{\rm i}(\k,\x)}.
$$
Note that condition $(iv)$ implies that
\begin{equation}\label{1}
  f=\sum_{\j\in \Z_+^d}\d_\j(f)
\end{equation}
with convergence in the sense of distributions.

Now we are ready to introduce the Triebel--Lizorkin and Besov spaces.

\begin{definition}\label{d3}
Let $\phi=(\phi_j)_{j\in \Z_+}\in \Phi(\R)$ and $r\in \R$.

\begin{itemize}
  \item[$1)$] For $1\le p<\infty$ and $1\le \t\le \infty$, the Triebel--Lizorkin space $\mathbf{F}_{p,\t}^r(\T^d)$ is defined as the collection of all $f\in L_1(\T^d)$ such that
$$
\|f\|_{\mathbf{F}_{p,\t}^r}:=\|2^{r|\j|_1}\d_\j(f)\|_{L_p(\T^d,\ell_\t)}<\infty.
$$

  \item[$2)$] For $1\le p\le\infty$ and $1\le \t\le \infty$, the Besov space $\mathbf{B}_{p,\t}^r(\T^d)$ is defined as the collection of all $f\in L_1(\T^d)$ such that
\begin{equation}\label{defb}
  \|f\|_{\mathbf{B}_{p,\t}^r}:=\|2^{r|\j|_1}\d_\j(f)\|_{\ell_\t(\T^d,L_p)}<\infty.
\end{equation}
\end{itemize}
\end{definition}

Recall that in the case $\t=2$ and $1<p<\infty$ the space $\mathbf{F}_{p,\t}^r(\T^d)$, $r\ge 0$, coincides with the Sobolev space of dominating mixed smoothness $\mathbf{W}_p^r(\T^d)$ with the norm given by
$$
\|f\|_{\mathbf{W}_p^r}:=\bigg\|\sum_{\k\in\Z^d}\prod_{i=1}^d (1+|k_i|^2)^\frac{r}{2}\h f(\k){e}^{{\rm i}(\k,\x)}\bigg\|_p.
$$

Recall also the following basic embedding (see, e.g.,~\cite{ST87} and~\cite{N75}):

\begin{itemize}
  \item[$1)$] if $1\le p\le \infty$ ($p<\infty$ for $F$-spaces) and $0<\t\le \infty$, $r>1/p$, then
\begin{equation}\label{e1}
  \mathbf{F}_{p,\t}^r(\T^d)\hookrightarrow C(\T^d)\quad\text{and}\quad \mathbf{B}_{p,\t}^r(\T^d)\hookrightarrow C(\T^d);
\end{equation}
  \item[$2)$] if $1\le p<\infty$, $0< \t\le \infty$, and $r\in \R$, then
\begin{equation}\label{e2}
  \mathbf{B}_{p,\min\{p,\t\}}^r(\T^d)\hookrightarrow \mathbf{F}_{p,\t}^r(\T^d)\hookrightarrow \mathbf{B}_{p,\max\{p,\t\}}^r(\T^d);
\end{equation}
  \item[$3)$] if $1\le p\le \infty$ ($p<\infty$ for $F$-spaces),  $0<\t_1<\t_2\le \infty$, and $r\in \R$, then
\begin{equation}\label{e3}
  \mathbf{F}_{p,\t_1}^r(\T^d)\hookrightarrow \mathbf{F}_{p,\t_2}^r(\T^d)\quad\text{and}\quad \mathbf{B}_{p,\t_1}^r(\T^d)\hookrightarrow \mathbf{B}_{p,\t_2}^r(\T^d);
\end{equation}
  \item[$4)$] if $1\le p<q<\infty$,  $0<\t,\nu \le \infty$, $r_1,r_2\in \R$, $r_2<r_1$, and $r_1-1/p=r_2-1/q$, then
\begin{equation}\label{e4}
  \mathbf{F}_{p,\t}^{r_1}(\T^d)\hookrightarrow \mathbf{F}_{q,\nu}^{r_2}(\T^d)\quad\text{and}\quad \mathbf{B}_{p,\t}^{r_1}(\T^d)\hookrightarrow \mathbf{B}_{q,\t}^{r_2}(\T^d).
\end{equation}
\end{itemize}

The following two lemmas were proved in~\cite{BU17}.

\begin{lemma}\label{llpb}
Let $1\le p,\t\le \infty$, $r>0$, and let $(t_\j)_{\j\in \Z_+^d}$ be such that $t_\j\in \mathcal{T}_{2^{\j}}^d$ and $\|2^{r|\j|_1} t_\j\|_{\ell_\t(\T^d,L_p)}<\infty$. Then $f=\sum_{\j\in \Z_+^d} t_\j$ belongs to $\mathbf{B}_{p,\t}^r(\T^d)$ and
\begin{equation}\label{llpb1}
  \|f\|_{\mathbf{B}_{p,\t}^r}\lesssim \|2^{r|\j|_1} t_\j\|_{\ell_\t(\T^d,L_p)}.
\end{equation}
\end{lemma}

\begin{lemma}\label{llpf}
Let $1\le p<\infty$, $1\le \t\le \infty$, $r>0$, and let $(t_\j)_{\j\in \Z_+^d}$ be such that $t_\j\in \mathcal{T}_{2^{\j}}^d$ and $\|2^{r|\j|_1} t_\j\|_{L_p(\T^d,\ell_\t)}<\infty$. Then $f=\sum_{\j\in \Z_+^d} t_\j$ belongs to $\mathbf{F}_{p,\t}^r(\T^d)$ and
\begin{equation}\label{llpf1}
  \|f\|_{\mathbf{F}_{p,\t}^r}\lesssim \|2^{r|\j|_1} t_\j\|_{L_p(\T^d,\ell_\t)}.
\end{equation}
\end{lemma}

\subsection{Maximal functions}\label{sec3.2}
For any $f\in L_1^{\rm loc}(\T^d)$ and $i=1,\dots,d$, we define the Hardy--Littlewood maximal function by
$$
\mathcal{M}_if(\x):=\sup_{h>0}\frac1{2h}\int_{x_i-h}^{x_i+h}|f(x_1,\dots,x_{i-1},t,x_{i+1},\dots,x_d)|dt.
$$
The following maximal inequality can be found in~\cite[4.1.2]{U06}:
{\it Let $1<p<\infty$, $1<\t\le\infty$, and $(f_\j)_{\j\in\Z_+^d}\subset L_p(\T^d,\ell_\t)$, then
\begin{equation}\label{m1}
  \|\mathcal{M}_i f_\j\|_{L_p(\T^d,\ell_\t)}\lesssim \|f_\j\|_{L_p(\T^d,\ell_\t)},\quad i=1,\dots,d,
\end{equation}
where the constant in $\lesssim$ is independent of $(f_\j)_{\j\in\Z_+^d}$.}

We also need the Peetre maximal function. For each $f\in C(\T^d)$, $a>0$, $b\ge 0$, and $i=1,\dots,d$, we define this maximal function by
$$
\mathcal{P}_{b,a,i}f(\x):=\sup_{y\in \R}\frac{|f(x_1,\dots,x_{i-1},x_i+y,x_{i+1},\dots,x_d)|}{(1+b|y|)^a}.
$$
The proof of the following inequality can be found in~\cite[4.1.4]{U06}: {\it Let $1\le p<\infty$, $1\le \t\le \infty$, $a>\max\{1/p,1/\t\}$, and let $(t_\j)_{\j\in \Z_+^d}\in L_p^\mathcal{T}(\T^d,\ell_\t)$.
Then
\begin{equation}\label{m2}
  \|\mathcal{P}_{2^{j_i},a,i} t_\j\|_{L_p(\T^d,\ell_\t)}\lesssim \|t_\j\|_{L_p(\T^d,\ell_\t)},\quad i=1,\dots,d,
\end{equation}
where the constant in $\lesssim$ is independent of $(t_\j)_{\j\in\Z_+^d}$.}

\subsection{Fourier multipliers}\label{sec3.3}
The main results of these papers are given in terms of Fourier multipliers. Recall that the sequence $\lambda=\(\lambda_k\)_{\k\in \Z^d}$ is called a Fourier multiplier in $L_p(\T^d)$ if for every function $f\in L_p(\T^d)$ the series
$$
\sum_{\k\in \Z^d} \lambda_\k \h  f(\k){e}^{{\rm i}(\k,\x)}
$$
is the Fourier series of a certain function $\lambda f \in L_p(\T^d)$ and
$$
\Vert \lambda\Vert_{M_p(\T^d)}=\sup_{\Vert f\Vert_p\le 1} \Vert \lambda f \Vert_p.
$$

Consider the sequences of Fourier multipliers of the form $(g(2^{-j}\k))_{\k\in \Z^d}$, where $g$ is a certain continuous function and $j\in \Z_+$. There are various sufficient conditions for such type of multipliers, see, e.g.,~\cite[Ch. IV]{SW}, \cite[Ch.~2 and~3]{GrI}, \cite{LST}, \cite{K14}. Here, we recall one simple condition, which follows from de Leeuw's theorem  (see, e.g.,~\cite[Ch.~VII, 3.8]{SW}) and Beurling’s-type condition of belonging to Wiener's algebra (see, e.g.,~\cite[Theorem~6.3]{LST}).

\begin{lemma}\label{lmult-}
Let $g\in L_2(\R^d)$ and $\frac{\partial^{r}}{\partial \xi_i^r} g\in L_2(\R^d)$, $i=1,\dots,d$, where $r>{d}/{2}$. Then $\sup_j \|(g(2^{-j}\k))_{\k}\|_{M_p(\T^d)}<\infty$.
\end{lemma}

Now let us give a definition of Fourier multipliers in $L_p(\T^d,\ell_\t)$. Let $\mu=(\mu_\j)_{\j\in \Z_+^d}$ be a sequence of bounded functions on $\R^d$. We say that $\mu$ is a Fourier multiplier in $L_p(\T^d,\ell_\t)$ ($\mu \in M_p(\T^d,\ell_\t)$) if there exists a constant $c=c(\mu,p,\t,d)$ such that
$$
\bigg\|\sum_{\k\in \Z^d}\mu_\j(\k)\h{t_\j}(\k){e}^{{\rm i}(\k,\x)}\bigg\|_{L_p(\T^d,\ell_\t)}\le c\|t_\j\|_{L_p(\T^d,\ell_\t)}
$$
for all sequences $(t_\j)_\j \in L_p^\mathcal{T}(\T^d,\ell_\t)$.

Recall one sufficient condition for $M_p(\T^d,\ell_\t)$. To formulate it, we need the definition of the Sobolev space $\mathbf{H}_2^\kappa (\R^d)$. We say that a function $f\in L_2(\R^d)$ belongs to $\mathbf{H}_2^\kappa (\R^d)$, $\kappa>0$, if
$$
\|f\|_{\mathbf{H}_2^\kappa (\R^d)}=\bigg(\int_{\R^d}(1+|\xi_1|^2)^\kappa\dots (1+|\xi_d|^2)^\kappa |\mathcal{F} f({\bm \xi})|^2d{\bm \xi}\bigg)^{1/2}<\infty,
$$
where $\mathcal{F}$ denoted the Fourier transform of $f$ in $L_2(\R^d)$.

\begin{lemma}\label{lmult} \emph{(See~\cite{SU07}.)}
  Let $1\le p<\infty$ and $1\le \t\le \infty$. If the sequence $\mu=(\mu_\j)_{\j\in \Z_+^d}$ in $\mathbf{H}_2^\kappa (\R^d)$ is such that
\begin{equation}\label{mult1}
  \sup_{\j\in \Z_+^d} \|\mu_\j(2^\j \cdot)\|_{\mathbf{H}_2^\kappa (\R^d)}<\infty\quad\text{for some} \quad \kappa>\max\{1/p,1/\t\}+1/2,
\end{equation}
then $\mu \in M_p(\T^d,\ell_\t)$.
\end{lemma}

\subsection{Inequalities for trigonometric polynomials}\label{sec3.4}
Recall the classical Nikolskii inequality (see, e.g.,~\cite[3.4.3]{N75}):

\begin{lemma}\label{lNik}
Let $1\le p\le q\le \infty$ and $t\in \mathcal{T}_\j^d$. Then
\begin{equation}\label{Nik}
  \|t\|_q\lesssim 2^{|\j|_1(\frac1p-\frac1q)}\|t\|_p.
\end{equation}
\end{lemma}

The proof of the next lemma can be found, e.g., in~\cite[p.~25]{T86}.

\begin{lemma}\label{lsum2}
Let $1\le p<q<\infty$ and $t_\j\in \mathcal{T}_\j^d$. Then
\begin{equation}\label{sum2}
  \bigg\|\sum_{\j} t_\j\bigg\|_q\lesssim \(\sum_{\j}\(2^{|\j|_1(\frac1p-\frac1q)}\|t_\j\|_p\)^q\)^{1/q}.
\end{equation}
\end{lemma}

We will also need the following analogue of Bernstein's inequality for trigonometric polynomials in the space $L_p(\T^d,\ell_\t)$. Recall that the fractional derivative of $f\in L_1(\T^d)$ in the sense of Weyl is given by
$$
\frac{\partial^s}{\partial x_i^s} f(\x)\sim \sum_{\k} ({\rm i}k_i)^s \h f(\k) e^{{\rm i}(\k,\x)}, \quad ({\rm i}k_i)^s=|k_i|^se^{\frac{{\rm i}\pi s}{2}\sign k_i},\quad i=1,\dots,d.
$$

\begin{lemma}\label{lber}
  Let $1\le p<\infty$, $1\le \t\le \infty$, and $s\in \N$ or $s> \max\{1/p,1/\t\}$. Then, for every sequence $(t_\j)_{\j\in \Z_+^d}\in L_p^\mathcal{T}(\T^d,\ell_\t)$, we have
\begin{equation}\label{ber1}
  \bigg\|\frac{\partial^s}{\partial x_i^s} t_\j\bigg\|_{L_p(\T^d,\ell_\t)}\lesssim \| 2^{s j_i} t_\j \|_{L_p(\T^d,\ell_\t)},\quad i=1,\dots,d.
\end{equation}
\end{lemma}

\begin{proof}
Denoting $\eta_i({\bm \xi})=({\rm i}\xi_i)^s \prod_{\nu=1}^d  v(\xi_\nu)$, where $v\in C^\infty(\R)$ and $v(\xi)=1$ for $|\xi|\le 1$ and $v(\xi)=0$ for $|\xi|\ge 2$, we get
$$
\frac{\partial^s}{\partial x_i^s} t_\j(\x)=2^{sj_i}\sum_{\k} \eta_i(2^{-\j}\k)\h{t_\j}(\k) e^{{\rm i}(\k,\x)}.
$$
Consider the functions $\eta_{\j,i}=\eta_i(2^{-\j}\cdot)$. If $s\in \N$, then $\eta_i$ is a Schwartz function and hence condition~\eqref{mult1} obviously holds for $(\eta_{\j,i})_{\j}$. If $s>\max\{1/p,1/\t\}$, then~\eqref{mult1} holds since $|\mathcal{F}(\eta_i)(\xi)|\lesssim \frac{1}{(1+|\xi_1|)^{s+1}\dots(1+|\xi_d|)^{s+1}}$, see, e.g.,~\cite{RS01}. Thus, applying Lemma~\ref{lmult}, we obtain~\eqref{ber1}.
\end{proof}

\subsection{Lower estimates for approximation by linear operators}\label{sec3.5}
The following general result was obtained in~\cite[Theorems~5.4 and~3.1 of Chapter~2]{T86}. To formulate it, we denote by
$\mathcal{Q}_n$ a step hyperbolic cross given by
$$
\mathcal{Q}_n=\cup_{|\j|_1\le n} \{\k\in \Z^d\,:\, [2^{j_i-1}]\le |k_i|<2^{j_i},\quad i=1,\dots,d\}
$$
and by $\mathcal{T}(\mathcal{Q}_n)$ we denote the set of all trigonometric polynomials with frequencies in $\mathcal{Q}_n$.

\begin{lemma}\label{lbel0}
Let $1<p<\infty$, $r>1/p$, and let $L_n$, $n\in \N$, be a bounded linear operator from $\mathbf{W}_{p}^r(\T^d)$ to $\mathcal{T}(\mathcal{Q}_n)$. Then
\begin{equation}\label{bel0}
  \sup_{f\in U\mathbf{W}_{p}^r}\|f-{L}_n(f)\|_\infty \gtrsim 2^{(r-\frac1p)n}n^{(d-1)(1-\frac1p)}.
\end{equation}
\end{lemma}

In the case of Besov space, we have the following lower estimates.

\begin{lemma}\label{lbel}
Let $1\le p<q\le\infty$, $1\le \t\le \infty$, $r>1/p-1/q$, and let $L_n$, $n\in \N$, be a bounded linear operator from $\mathbf{B}_{p,\t}^r(\T^d)$ to $\mathcal{T}(\mathcal{Q}_n)$. Then
\begin{equation}\label{bel1}
  \sup_{f\in U\mathbf{B}_{p,\t}^r}\|f-{L}_n(f)\|_q \gtrsim \left\{
                                                    \begin{array}{ll}
                                                        2^{-(r-\frac1p+\frac1q)n}n^{(d-1)(\frac1q-\frac1\t)_+}, & \hbox{$q<\infty$,} \\
                                                        2^{-(r-\frac1p)n}n^{(d-1)(1-\frac1\t)}, & \hbox{$q=\infty$.}
                                                    \end{array}
                                                              \right.
\end{equation}
\end{lemma}

\begin{proof}
In the case $q=\t=\infty$, inequality~\eqref{bel1} was established in~\cite[Chapter~2, Theorem~5.7]{T86}. The case $q=\infty$ and $1\le \t<\infty$  was considered in~\cite{R04}. We prove~\eqref{bel1} for the case $q<\infty$. Similarly to paper~\cite{R04}, we follow the proof of Theorem~5.7 in~\cite[Chapter~2]{T86}.
Denote
$$
U_n(f)(\x)=(2\pi)^{-d}\int_{\T^d} I_{-{\bm \tau}} L_n(I_{\bm \tau} f)(\x)d{\bm \tau},
$$
where $I_{{\bm \tau}}f(\x)=f(\x+{\bm \tau})$, ${\bm \tau}=(\tau_1,\dots,\tau_d)\in \R^d$. It is easy to see that
$$
U_n\big({e}^{{\rm i}(\k,\cdot)}\big)(\x)=\left\{
                                         \begin{array}{ll}
                                           c_{\k,n}{e}^{{\rm i}(\k,\x)}, & \hbox{$\k\in \mathcal{Q}_n$,} \\
                                           0, & \hbox{$\k\not\in \mathcal{Q}_n$.}
                                         \end{array}
                                       \right.
$$
Let $f\in U\mathbf{B}_{p,\t}^r$. Using Minkowski's inequality, we derive
\begin{equation}\label{bel2}
  \begin{split}
     \|f-U_n(f)\|_q&=\bigg\|(2\pi)^{-d}\int_{\T^d} I_{-{\bm \tau}}\big( I_{{\bm \tau}}f- L_n(I_{\bm \tau} f)\big)d{\bm \tau}\bigg\|_q\\
&\le  (2\pi)^{-d}\int_{\T^d} \| I_{-{\bm \tau}}\big( I_{{\bm \tau}}f- L_n(I_{\bm \tau} f)\big) \|_q d{\bm \tau}\\
&=  (2\pi)^{-d}\int_{\T^d} \|  I_{{\bm \tau}}f- L_n(I_{\bm \tau} f)\|_q d{\bm \tau}\le\sup_{f\in U\mathbf{B}_{p,\t}^r}\|f-{L}_n(f)\|_q.
  \end{split}
\end{equation}
Next, we denote
\begin{equation}\label{phij}
  \Phi_\j(\x)=\sum_{\k\in \Z^d}\phi_{j_1}(k_1)\dots \phi_{j_d}(k_d)e^{{\rm i}(\k,\x)}.
\end{equation}
It is known (see, e.g.,~\cite{SU06}) that
\begin{equation}\label{bel3}
  \|\Phi_\j\|_s\asymp 2^{(1-\frac1s)|\j|_1},\quad 1\le s\le \infty,
\end{equation}
and
\begin{equation}\label{bel3}
  \bigg\|\sum_{|\j|_1=m}\Phi_\j\bigg\|_s\asymp \left\{
                                              \begin{array}{ll}
                                                2^{(1-\frac1s)m}m^{\frac{d-1}{s}}, & \hbox{$1\le s<\infty$,} \\
                                                2^m m^{d-1}, & \hbox{$s=\infty$.}
                                              \end{array}
                                            \right.
\end{equation}
Consider the function
$$
f(\x)=c2^{-(r+1-\frac1p)n}n^{-\frac{d-1}{\t}}\sum_{|\j|_1=n+4}\Phi_\j(\x).
$$
It follows from~\eqref{bel3} that we can choose the constant $c$ such that $f\in U\mathbf{B}_{p,\t}^r$.
Thus, taking into account that $U_n(f)=0$ and using~\eqref{bel3}, we derive
\begin{equation*}
  \begin{split}
      \|f-U_n(f)\|_q=\|f\|_q\asymp 2^{-(r+1-\frac1p)n}n^{-\frac{d-1}{\t}}2^{(1-\frac1q)n}n^{\frac{d-1}{q}}\asymp 2^{-(r-\frac1p+\frac1q)n}n^{(d-1)(\frac1q-\frac1\t)}.
   \end{split}
\end{equation*}
This together with~\eqref{bel2} proves~\eqref{bel1} for $q<\t\le \infty$.

To prove~\eqref{bel1} in the case $q\ge \t$,  it is enough to take $f(\x)=c 2^{-n(r+1-\frac1p)}\Phi_{n+4,0,\dots,0}(\x)$. Then, as above, we have that $f\in U\mathbf{B}_{p,\t}^r$ for an appropriate  constant $c$ and
\begin{equation*}
  \begin{split}
      \|f-U_n(f)\|_q=\|f\|_q\asymp 2^{-(r-\frac1p+\frac1q)n},
   \end{split}
\end{equation*}
which together with~\eqref{bel2} proves the lemma.
\end{proof}

\begin{remark}
In some partial cases for $p,q$, and $\t$,  Lemma~\ref{lbel} follows from known estimates for the best approximation by trigonometric polynomials with frequencies in the hyperbolic cross. Recall that if $1<p<q<\infty$, $1\le \t\le \infty$, and $r>1/p-1/q$, then
\begin{equation*}
  2^{-(r-\frac1p+\frac1q)n}n^{(d-1)(\frac1q-\frac1\t)_+}\asymp \sup_{f\in U\mathbf{B}_{p,\t}^r}\inf_{t\in \mathcal{T}(\mathcal{Q}_n)}\|f-t\|_q\lesssim \sup_{f\in U\mathbf{B}_{p,\t}^r }\Vert f-L_n(f)\Vert_q,
\end{equation*}
see~\cite{D85} for the case $\t<\infty$ and~\cite{T86} for the case $\t=\infty$ (in the later case, the estimate is also valid for $p=1$).
\end{remark}

\subsection{Complex interpolation}\label{sec3.6}
The next lemma is a standard corollary from the complex interpolation results in~\cite[1.18]{Tr78} and~\cite[Section~1.5]{Udis}.

\begin{lemma}\label{li}
 Let $1\le p_0,p_1<\infty$, $1\le \t_0,\t_1< \infty$, and $r_0,r_1\in \R$. Let further $0<\tau<1$ and
$$
\frac1p=\frac{1-\tau}{p_0}+\frac{\tau}{p_1},\quad \frac1\t=\frac{1-\tau}{\t_0}+\frac{\tau}{\t_1},\quad\text{and}\quad r=(1-\tau)r_0+\tau r_1.
$$
Suppose that $T$ is a linear bounded operator that maps $\mathbf{F}_{p_0,\t_0}^{r_0}(\T^d)$ to $L_{p_0}(\T^d,\ell_{\t_0}^{r_0})$ with  norm $A_0$ and $\mathbf{F}_{p_1,\t_1}^{r_1}(\T^d)$ to $L_{p_1}(\T^d,\ell_{\t_1}^{r_1})$ with norm $A_1$. Then $T$ maps $\mathbf{F}_{p,\t}^{r}(\T^d)$ to $L_{p}(\T^d,\ell_{\t}^{r})$ with norm at most $A_0^{1-\tau}A_1^\tau$.
\end{lemma}

\subsection{Some relations for the Smolyak  algorithms and quasi-interpolation operators}\label{sec3.7}

The following general result is presented in~\cite[Proposition~4.2]{DTU18}, see also~\cite{AT97} for the case $\t=\infty$ and~\cite{SU07} for some special cases of operators $(Y_j)_{j\in \Z_+}$.
\begin{lemma}\label{pr1}
Let $1\le p \le \infty$ and $0\le b<r<s$.
Suppose that for a family of univariate linear operators $Y=\(Y_j\)_{j\in \Z_+}$,
which are defined on the Sobolev space $W_p^s(\T)$, the following properties hold:

\begin{enumerate}
  \item[1)] For any $f\in W_p^s(\T)$ and $j\in \Z_+$, we have
    $$
     \Vert f-Y_j(f)\Vert_{L_p(\T)} \lesssim 2^{-sj}\Vert f\Vert_{W_p^s(\T)}.
    $$

  \item[2)] For any trigonometric polynomial $t$ of degree $2^u$, $u\ge j$, we have
  $$
   \Vert Y_j(t)\Vert_{L_p(\T)} \lesssim 2^{b(u-j)}\Vert t\Vert_{L_p(\T)}.
  $$
\end{enumerate}
Then every $f\in \mathbf{B}_{p,\t}^r(\T^d)$ admits the representation
\begin{equation}\label{repB}
  f=\sum_{\j\in \Z_+^d}\D_\j^Y(f)
\end{equation}
with unconditional convergence in $L_p(\T^d)$ and the following inequalities hold:
\begin{equation}\label{b1}
  \|2^{r|\j|_1}\Delta_\j^Y(f)\|_{\ell_\t(\T^d,L_p)}\lesssim \|f\|_{\mathbf{B}_{p,\t}^r(\T^d)},
\end{equation}
\begin{equation}\label{b2}
    \Vert f-T_n^Y(f)\Vert_{L_p(\T^d)} \lesssim 2^{-r n}n^{(d-1)(1-1/\t)}\Vert f\Vert_{\mathbf{B}_{p,\t}^r(\T^d)}.
\end{equation}
\end{lemma}

It worths noting that the $L_p$-error estimate given in~\eqref{b2} is sharp for a wide class of sampling operators of type~\eqref{I} (cf. estimates~\eqref{I2}, see also~\cite{DT15} and~\cite{T15} for more general results). At the same time, it is not sharp in general for the Smolyak algorithm generated by an appropriate convolution operator (cf.~\eqref{S2}).

In this paper, we are interested in applications of Lemma~\ref{pr1} to studying properties of the Smolyak algorithm constructed by means of the general quasi-interpolation operators $Q_j(\cdot,\vp,\w\vp)$ (cf.~\eqref{Q0}), where $\w\vp=(\w\vp_j)_{j\in \Z_+}$ is a family of integrable functions satisfying certain additional conditions given in the next lemma.  In particular, we are interested in constructions involving the Kantorovich-type operators of form~\eqref{K1}, which are "intermediate" operators between the de la Vall\'ee Poussin means $V_j$ and the sampling operators $I_j$.

For our purposes, we will use the following special norms for functions $\w\vp_j \in L_q(\T)$ and $j\in \Z_+$:
$$
\Vert \w\vp_j\Vert_{\mathcal{L}_{q,j}(\T)}=\(2^j\int_{2^{-j}\T} \bigg(\frac1{2^j}\sum_{k\in A_j} |\w\vp_j(x-x_k^j)| \bigg)^q {\rm d}x\)^\frac1q\quad\text{if}\quad 1\le q<\infty
$$
and
$$
\Vert \w\vp_j\Vert_{\mathcal{L}_{\infty,j}(\T)}=\frac1{2^j}\sup_{x\in \R}\sum_{k\in A_j} |\w\vp_j(x-x_k^j)|\quad\text{if}\quad q=\infty.
$$

The following lemma is a key result for our further investigations.

\begin{lemma}\label{lwc} {\sc (See~\cite{KP20}.)}
  Let $1\le p\le \infty$ and $s>0$. Suppose that the sequences of functions
$\vp=\(\vp_j\)_{j}$ and $\w\vp=\(\w\vp_j\)_{j}$ satisfy the following conditions:
\begin{itemize}
  \item[$1)$] $\vp_j\in \mathcal{T}_j^1$ and $\sup_j \|(\h{\vp_j}(k))_k\|_{M_p(\T)}<\infty$;
  \item[$2)$] $\sup_j \|\w\vp_j\|_{\mathcal{L}_{p',j}(\T)}<\infty$;
  \item[$3)$] there exists $\d>0$ such that
\begin{equation}\label{wc}
  \sup_j\bigg\|\bigg(\frac{1-\h{\vp_j}(k)\h{\w\vp_j}(k)}{(2^{-j}{\rm i}k)^s}\phi_0(\d 2^{-j}k)\bigg)_{\!\!k}\bigg\|_{M_p(\T)}<\infty.
\end{equation}
\end{itemize}
Then, for every $f\in W_p^s(\T)$ and $j\in \Z_+$, we have
\begin{equation}\label{wc1}
  \|f-Q_j(f,\vp,\w\vp)\|_{L_p(\T)}\lesssim 2^{-sj}\|f\|_{W_p^s(\T)}.
\end{equation}
\end{lemma}

\begin{remark}\label{rem1}
In~\cite{KP20}, it was also proved an analogue of Lemma~\ref{lwc} in the case $\w\vp=\(\w\vp_j\)_{j}$  is a sequence of distributions (continuous linear functionals on $C^\infty(\T)$) satisfying the following property: there exist constants $c>0$ and $\s\ge 0$ such that for any trigonometric polynomial $t\in \mathcal{T}_u^1$, $u\ge j$, and $j\in\Z_+$ there holds
\begin{equation}\label{dis}
  \|\w\vp_j*t\|_{L_p(\T)}\le c2^{\s (u-j)}\|t\|_{L_p(\T)}.
\end{equation}
In particular,  if $\vp=\(\vp_j\)_{j}$ and $\w\vp=\(\w\vp_j\)_{j}$ satisfy conditions 1) and 3) of Lemma~\ref{lwc} and~\eqref{dis} holds, then for every $f\in W_p^s(\T)$ with $s>\s+1/p$ and $j\in \Z_+$, we have
\begin{equation}\label{wc1+}
  \|f-Q_j(f,\vp,\w\vp)\|_{L_p(\T)}\lesssim 2^{-sj}\|f\|_{W_p^s(\T)}.
\end{equation}
\end{remark}

The next lemma will be useful to establish upper estimates of approximation of functions from $\mathbf{F}_{p,\t}^r(\T^d)$.

\begin{lemma}\label{lf}
  Let $j, \ell \in \Z_+$, $a>0$, $1/2<\l\le 1$, and let $\vp=(\vp_j)_j$ and $\w\vp=(\w\vp_j)_j$ satisfy the following conditions:
\begin{itemize}
  \item[$1)$] there exists a constant $C>0$ such that $|\vp_j(x)|\le \frac{C2^j}{(1+2^j|x|)^2}$ for all $j\in \Z_+$ and $x\in \R$;
  \item[$2)$] there exist constants $c>0$ and $\tau>0$ such that
\begin{equation}\label{lf1}
  |(\w\vp_{j}*g)(x)|\le c\sup_{{|y|\le \tau 2^{-j}}}|g(x+y)|\quad\text{for all}\quad g\in C(\T).
\end{equation}
\end{itemize}
Then, for every $f\in C(\T)$, we have
\begin{equation}\label{lf2}
|Q_j(f,\vp,\w\vp)(x)|\lesssim 2^{a\ell}\(\mathcal{M}_1(\mathcal{P}_{2^{j+\ell},a,1} f)^\l(x)\)^{1/\l},
\end{equation}
where the constant in $\lesssim$ is independent of $\ell$, $j$, $x$, and $f$.
\end{lemma}

\begin{proof}
Consider the sampling-type operator $\w I_j$ defined by
$$
\w I_j(g)(x)=2^{-j}\sum_{k\in {A}_j} g(x_k^j)\vp_j(x-x_k^j).
$$
It follows from~\cite[the proof of Proposition~5.6]{BU17} that for every $g\in C(\T)$ one has
\begin{equation*}
|\w I_j(g)(x)|\lesssim \(\mathcal{M}_1(\mathcal{P}_{2^{j},a,1} g)^\l(x)\)^{1/\l}.
\end{equation*}
Applying this inequality to $g=f*\w\vp_j$, we get
\begin{equation}\label{lf4}
|Q_j(f,\vp,\w\vp)(x)|=|\w I_j(f*\w\vp_j)(x)|\lesssim \(\mathcal{M}_1\big(\mathcal{P}_{2^{j},a,1} (f*\w\vp_j)\big)^\l(x)\)^{1/\l}.
\end{equation}
Next, using~\eqref{lf1}, we derive
\begin{equation}\label{lf5}
\begin{split}
   \mathcal{P}_{2^{j},a,1} (f*\w\vp_j)(x)&=\sup_{y\in\R}\frac{|(f*\w\vp_j) (x+y)|}{(1+2^{j}|y|)^a}\\
&\le c\sup_{y\in\R} \sup_{{|t|\le \tau 2^{-j}}}\frac{|f(x+y+{t})|}{(1+2^{j}|y|)^a}\\
&= c\sup_{y\in\R} \sup_{{|z-y|\le \tau 2^{-j}}}\frac{|f(x+z)|}{ (1+2^{j}|z|)^a}\(\frac{1+2^{j}|z|}{1+2^{j}|y|}\)^a\\
&\le c(1+\tau)^{a} \mathcal{P}_{2^j,a,1}f(x)\le c(1+\tau)^{a} 2^{a\ell}\mathcal{P}_{2^{j+\ell},a,1}f(x).
\end{split}
\end{equation}
Finally, combining~\eqref{lf4} and~\eqref{lf5}, we prove the lemma.
\end{proof}

\section{Main results for $B$-spaces}\label{sec4}

\subsection{Littlewood--Paley type characterizations}

\begin{theorem}\label{tb1}
Let $1\le p, \t\le \infty$, and let $Q=\(Q_j(\cdot,\vp,\w\vp)\)_{j\in \Z_+}$, where
the sequences
$\vp=\(\vp_j\)_{j}$ and $\w\vp=\(\w\vp_j\)_{j}$ satisfy the following conditions:
\begin{itemize}
  \item[$1)$] $\vp_j\in \mathcal{T}_j^1$ and $\sup_j \|(\h{\vp_j}(k))_k\|_{M_p(\T)}<\infty$;
  \item[$2)$] $\sup_j \|\w\vp_j\|_{\mathcal{L}_{p',j}}<\infty$;
  \item[$3)$] compatibility condition~\eqref{wc} of order $s>0$ holds.
\end{itemize}
Then every $f\in \mathbf{B}_{p,\t}^r(\T^d)$, $0<r<s$, admits the representation
\begin{equation}\label{tlpb.0}
  f=\sum_{\j\in \Z_+^d}\D_\j^Q(f)
\end{equation}
with unconditional convergence in $L_p(\T^d)$ and
\begin{equation}\label{tlpb.1}
\|2^{r|\j|_1}\Delta_\j^Q(f)\|_{\ell_\t(\T^d,L_p)}\asymp \|f\|_{\mathbf{B}_{p,\t}^r}.
\end{equation}
\end{theorem}

\begin{proof}
The estimate from below in~\eqref{tlpb.1} follows directly from Lemma~\ref{llpb}. To obtain the representation~\eqref{tlpb.0} and the estimate from above in~\eqref{tlpb.1}, we apply Lemmas~\ref{pr1} and~\ref{lwc}.
\end{proof}

\subsection{Error estimates}

\begin{theorem}\label{tb2}
  Let $1\le p, q, \t\le \infty$, $(1/p-1/q)_+<r<s$, and let $Q=\(Q_j(\cdot,\vp,\w\vp)\)_{j\in \Z_+}$, where
$\vp=\(\vp_j\)_{j}$ and $\w\vp=\(\w\vp_j\)_{j}$ satisfy conditions 1)--3) of Theorem~\ref{tb1}.
Then, for every $f\in \mathbf{B}_{p,\t}^r(\T^d)$ and $n\in \N$, we have
\begin{equation}\label{t1.1}
    \Vert f-T_n^Q(f)\Vert_q \lesssim \Vert f\Vert_{\mathbf{B}_{p,\t}^r}
\left\{
    \begin{array}{ll}
     2^{-r n}n^{(d-1)(1-\frac1\t)}, & \hbox{$q\le p$,} \\
    2^{-(r-\frac1p+\frac1q) n}n^{(d-1)(\frac1q-\frac1\t)_+}, & \hbox{$p<q<\infty$,} \\
    2^{-(r-\frac1p) n}n^{(d-1)(1-\frac1\t)}, & \hbox{$p<q=\infty$.}
    \end{array}
\right.
\end{equation}
\end{theorem}

\begin{proof}
By Theorem~\ref{tb1}, we have
\begin{equation}\label{im1}
  \|f-T_n^Q(f)\|_q=\bigg\|\sum_{|\j|_1>n}\D_\j^Q(f)\bigg\|_q.
\end{equation}
The rest of the proof is quite standard (see, e.g.,~\cite[Theorem~5.9]{DTU18}). For convenience of the reader, we present it in details for all cases of parameters $p$ and $q$.

1) The case $q\le p$. Using~\eqref{im1}, Minkowski's and H\"older's inequalities,
estimate
\begin{equation}\label{sum}
  \sum_{|\j|_1>n} 2^{-\rho|\j|_1}\asymp 2^{-\rho n}n^{{d-1}},\quad \rho>0,
\end{equation}
and~\eqref{tlpb.1}, we derive
\begin{equation*}
  \begin{split}
     &\|f-T_n^Q(f)\|_q\lesssim \sum_{|\j|_1>n}\|\D_\j^Q(f)\|_p\\
&\lesssim\(\sum_{|\j|_1>n} 2^{-r\t'|\j|_1}\)^{1/\t'} \(\sum_{|\j|_1>n} 2^{r\t|\j|_1}\|\D_\j^Q(f)\|_p^\t\)^{1/\t}\lesssim 2^{-r n}n^{(d-1)(1-\frac1\t)} \| f\|_{\mathbf{B}_{p,\t}^r}.
   \end{split}
\end{equation*}

2) The case $p<q<\infty$. Combining~\eqref{im1} and~\eqref{sum2},  we get
\begin{equation*}
  \begin{split}
     \|f-T_n^Q(f)\|_q&\lesssim \(\sum_{\j}\(2^{(\frac1p-\frac1q)|\j|_1}\|\D_\j^Q(f)\|_p\)^q\)^{1/q}:=I.
   \end{split}
\end{equation*}
If $\t>q$, then applying H\"older's inequality, \eqref{sum}, and~\eqref{tlpb.1}, we have
\begin{equation*}
  \begin{split}
     I&\lesssim\(\sum_{|\j|_1>n} 2^{\frac{q\t}{\t-q}(\frac1p-\frac1q-r)|\j|_1}\)^{1/q-1/\t} \(\sum_{|\j|_1>n} 2^{r\t|\j|_1}\|\D_\j^Q(f)\|_p^\t\)^{1/\t}\\
&\lesssim 2^{-(r-\frac1p+\frac1q) n}n^{(d-1)(\frac1q-\frac1\t)} \| f\|_{\mathbf{B}_{p,\t}^r}.
   \end{split}
\end{equation*}
If $\t\le q$, then using the triangle inequality and~\eqref{tlpb.1}, we derive
\begin{equation*}
  \begin{split}
     I&\lesssim 2^{-(r-\frac1p+\frac1q)n} \(\sum_{|\j|_1>n} 2^{r\t|\j|_1}\|\D_\j^Q(f)\|_p^\t\)^{1/\t}\lesssim 2^{-(r-\frac1p+\frac1q)n}  \| f\|_{\mathbf{B}_{p,\t}^r}.
   \end{split}
\end{equation*}

3) The case $p<q=\infty$. Using~\eqref{im1}, Nikolskii's inequality~\eqref{Nik}, H\"older's inequality, \eqref{sum}, and~\eqref{tlpb.1}, we obtain
\begin{equation*}
  \begin{split}
     &\|f-T_n^Q(f)\|_\infty\lesssim \sum_{|\j|_1>n}2^{|\j|_1/p}\|\D_\j^Q(f)\|_p\\
&\lesssim\(\sum_{|\j|_1>n} 2^{(1/p-r)\t'|\j|_1}\)^{1/\t'} \(\sum_{|\j|_1>n} 2^{r\t|\j|_1}\|\D_\j^Q(f)\|_p^\t\)^{1/\t}\lesssim 2^{-(r-\frac1p) n}n^{(d-1)(1-\frac1\t)} \| f\|_{\mathbf{B}_{p,\t}^r},
   \end{split}
\end{equation*}
which proves the theorem.
\end{proof}

In the next theorem, we show that under some additional conditions on the sequence $\w\vp=\(\w\vp_j\)_j$ estimate~\eqref{t1.1} is sharp.
In particular, we suppose that there exists $\xi\in \N$ and $\l\neq 0$ such that, for any $u\ge \xi+1$ and $j\in \Z_+$, the following equality holds:
\begin{equation}\label{cond}
  \h{\w\vp_{j}}(2^u)-\h{\w\vp_{j-1}}(2^u)=\left\{
                                            \begin{array}{ll}
                                              \l, & \hbox{$j=u-\xi$,}  \\
                                              0, & \hbox{$0\le j<u-\xi$.}
                                            \end{array}
                                          \right.
\end{equation}

An important example of such sequences $\(\w\vp_j\)_j$ is given by the following functions
$$
\w\vp_j(x)=2^{j+\s}\chi_{[-\frac{\pi}{2^{j+\s}},\frac{\pi}{2^{j+\s}}]}(x)\sim \sum_{k\in \Z}\frac{\sin 2^{-j-\s}\pi k}{2^{-j-\s}\pi k} e^{{\rm i}kx}.
$$
Here we have $\xi=\s-1$.

\begin{theorem}\label{tb3}
  Let $1\le p, q, \t\le \infty$, $(1/p-1/q)_+<r<s$, and let $Q=\(Q_j(\cdot,\vp,\w\vp)\)_{j\in \Z_+}$, where
$\vp=\(\vp_j\)_{j}$ and $\w\vp=\(\w\vp_j\)_{j}$ satisfy conditions 1)--3) of Theorem~\ref{tb1}. Suppose additionally that
$\h{\vp_j}(0)=1$, $j\in \Z_+$, and condition~\eqref{cond} holds for $\w\vp_j$.
Then
\begin{equation}\label{t2.1}
   \sup_{f\in U\mathbf{B}_{p,\t}^r }\|f-T_n^Q(f)\|_q \asymp
\left\{
    \begin{array}{ll}
     2^{-r n}n^{(d-1)(1-\frac1\t)}, & \hbox{$q\le p$,} \\
     2^{-(r-\frac1p+\frac1q) n}n^{(d-1)(\frac1q-\frac1\t)_+}, & \hbox{$p<q<\infty$,} \\
     2^{-(r-\frac1p) n}n^{(d-1)(1-\frac1\t)}, & \hbox{$p<q=\infty$.}
    \end{array}
\right.
\end{equation}
\end{theorem}

\begin{proof}
In view of Theorem~\ref{tb2}, we only need to establish the estimates from below.
In the case $1\le p<q\le \infty$, the corresponding estimates easily follow from Lemma~\ref{lbel}.
Let us consider the case $q\le p$. We follow the proof of Theorem~2 in~\cite{SU07}.  Recall that if $f$ belongs to the Wiener algebra, then the Fourier coefficients of $Q_j(f)$ can be given by the following formula:
$$
\h{Q_j (f)}(k)=\h{\vp_j}(k)\sum_{\ell\in \Z}\h{\w\vp_j}(k+\ell 2^j)\h{f}(k+\ell 2^j).
$$
Thus, denoting $c_k(g)=\h g(k)$, $k\in \Z$, and ${\rm e}_m(x)={e}^{{\rm i}mx}$, we get
\begin{equation}\label{t2.2}
  c_0({Q_j ({\rm e}_{2^u})})=\sum_{\ell\in \Z}c_{\ell 2^j}(\w\vp_j)c_{\ell 2^j}({\rm e}_{2^u})=\left\{
                                                                           \begin{array}{ll}
                                                                             c_{2^u}(\w\vp_j), & \hbox{$j\le u$,} \\
                                                                             0, & \hbox{otherwise,}
                                                                           \end{array}
                                                                         \right.
\end{equation}
which implies that
\begin{equation}\label{t2.3}
  c_0\(Q_j ({\rm e}_{2^u})-Q_{j-1} ({\rm e}_{2^u})\)=\left\{
                                           \begin{array}{ll}
                                               0, & \hbox{$u\le j-2$,} \\
                                             -c_{2^u}(\w\vp_{j-1}), & \hbox{$j-1=u$,} \\
                                              c_{2^u}(\w\vp_j-\w\vp_{j-1}), & \hbox{$j\le u$,} \\
                                              c_{2^u}(\w\vp_0), & \hbox{$j=0$.}
                                           \end{array}
                                         \right.
\end{equation}
Now, we consider the function
$$
f_n(\x)=\sum_{{}^{|\u|_1=n+d\xi}_{u_i\ge \xi+1}} e^{{\rm i}2^{u_1}x_1+\dots {\rm i}2^{u_d}x_d}.
$$
Denoting ${\rm e}_m^i(\x)=e_m(x_i)$, we derive
\begin{equation}\label{t2.5}
  \begin{split}
    c_0\(T_n^Q (f_n)\)&=\sum_{{}^{|\u|_1=n+d\xi}_{u_i\ge \xi+1}} c_0\(T_n^Q ({\rm e}_{2^{u_1}}^1\dots {\rm e}_{2^{u_d}}^d)\)\\
&=\sum_{{}^{|\u|_1=n+d\xi}_{u_i\ge \xi+1}} \sum_{|\j|_1\le n} \prod_{i=1}^d c_0\(Q_{j_i} ({\rm e}_{2^{u_i}})-Q_{j_i-1} ({\rm e}_{2^{u_i}})\).
  \end{split}
\end{equation}
It follows from~\eqref{t2.3} and condition~\eqref{cond} that
$$
\prod_{i=1}^d c_0\(Q_{j_i} ({\rm e}_{2^{u_i}})-Q_{j_i-1} ({\rm e}_{2^{u_i}})\)=0\quad\text{if}\quad j_i< u_i-\xi\quad\text{for some}\quad i=1,\dots,d.
$$
Thus, since $|\u|_1=n+d\xi$ and $|\j|_1\le n$, we need to consider only $\j$ with $j_i=u_i-\xi$, $i=1,\dots,d$. This together with~\eqref{t2.5} and~\eqref{cond} implies that
\begin{equation}\label{t2.6}
  \begin{split}
    c_0\(T_n^Q (f_n)\)=\sum_{{}^{|\u|_1=n+d\xi}_{u_i\ge \xi+1}} |\l|^d\asymp n^{d-1}.
  \end{split}
\end{equation}
Next, using~\eqref{defb} and~\eqref{sum}, it is not difficult to verify that $\|f_n\|_{\mathbf{B}_{p,\t}^r}\asymp 2^{rn}n^{(d-1)/\t}$. Thus, we can choose a positive constant $c$ such that the function $g_n=c2^{-rn}n^{-(d-1)/\t}f_n$ belongs to $U\mathbf{B}_{p,\t}^r$.
Then, applying~\eqref{t2.6}, we obtain
\begin{equation*}
  \begin{split}
     \sup_{f\in U\mathbf{B}_{p,\t}^r }\|f-T_n^Q(f)\|_q&\ge \|g_n-T_n^Q(g_n)\|_q \gtrsim|c_0\(g_n-T_n^Q(g_n)\)|\\
&=|c_0\(T_n^Q(g_n)\)|=c 2^{-rn}n^{-\frac{d-1}{\t}}|c_0\(T_n^Q(f_n)\)|\asymp 2^{-r n}n^{(d-1)(1-\frac1\t)},
   \end{split}
\end{equation*}
which proves~\eqref{t2.1} in the case $q\le p$.
\end{proof}

\section{Main results for $F$-spaces}\label{sec5}

\subsection{Littlewood--Paley type characterizations}

To formulate an analogue of Theorem~\ref{tb1} in the case of $F$-spaces, we will use the following notation:
for $s>0$ and the families of functions $\vp=\(\vp_j\)_{j\in \Z_+}$ and $\w\vp=\(\w\vp_j\)_{j\in \Z_+}$, we denote
$$
\mu_{\j,i}^{(s)}=\mu_{\j,i}^{(s)}(\vp,\w\vp)=\bigg(\frac{1-\h{\vp_{j_i}}(k_i)\h{\w\vp_{j_i}}(k_i)}{(2^{-j_i} {\rm i}k_i)^s}\bigg)_{\k\in \Z^d},\quad \j\in \Z_+^d,\quad i=1,\dots,d.
$$

\begin{theorem}\label{tf1}
  Let $1\le p,\t<\infty$ and let $Q=\(Q_j(\cdot,\vp,\w\vp)\)_{j\in \Z_+}$, where
the sequences
$\vp=\(\vp_j\)_{j\in \Z_+}$ and $\w\vp=\(\w\vp_j\)_{j\in \Z_+}$ satisfy the following conditions:
\begin{itemize}
  \item[$1)$] $\vp_j\in \mathcal{T}_j^1$, $\sup_j \|(\h{\vp_j}(k))_{k}\|_{M_p(\T)}<\infty$, and $\(\(\h{\vp_{j_i}}(k_i)\)_{\k}\)_{\j\in \Z_+^d} \in M_p(\T^d,\ell_\t)$ for each $i=1,\dots,d$;
  \item[$2)$] $\sup_j \|\w\vp_j\|_{\mathcal{L}_{p',j}(\T)}<\infty$ and there exist constants $c>0$ and $\tau>0$ such that
\begin{equation}\label{lf1+}
  |(\w\vp_{j}*g)(x)|\le c\sup_{{|y|\le \tau 2^{-j}}}|g(x+y)|\quad\text{for all}\quad g\in C(\T);
\end{equation}
  \item[$3)$] compatibility condition~\eqref{wc} of order $s>0$ holds and
$(\mu_{\j,i}^{(s)})_{\j\in \Z_+^d}\in M_p(\T^d,\ell_\t)$ for each $i=1,\dots,d$.
\end{itemize}
Then every $f\in \mathbf{F}_{p,\t}^r(\T^d)$, $0<r<s$, admits the representation
\begin{equation}\label{tlpf.0}
  f=\sum_{\j\in \Z_+^d}\D_\j^Q(f)
\end{equation}
with unconditional convergence in $L_p(\T^d)$ and
\begin{equation}\label{tlpf.1}
\|2^{r|\j|_1}\Delta_\j^Q(f)\|_{L_p(\T^d,\ell_\t)}\asymp \|f\|_{\mathbf{F}_{p,\t}^r}.
\end{equation}
\end{theorem}

\begin{proof}
Similar to the proof of Theorem~\ref{tb1}, we have that equality~\eqref{tlpf.0} follows easily from Lemmas~\ref{pr1} and~\ref{lwc} and the embedding $\mathbf{F}_{p,\t}^r(\T^d)\hookrightarrow \mathbf{B}_{p,\max\{p,\t\}}^r(\T^d)$, see~\eqref{e2}. The estimate from below in~\eqref{tlpf.1} follows directly from Lemma~\ref{llpf}.
To prove the estimate from above, we first consider the case $r>\max\{1/p,1/\t\}$.
Let $f\in \mathbf{F}_{p,\t}^r(\T^d)$ and $\j\in \Z_+^d$. By~\eqref{1}, we have
\begin{equation}\label{f2}
  f(\x)=\sum_{\elll\in \Z^d} \d_{\j+\elll}(f)(\x),
\end{equation}
where we set $\d_{\j}(f)=0$ for $\j\in \Z^d \setminus \Z_+^d$. In view of the embedding $\mathbf{F}_{p,\t}^r(\T^d)\hookrightarrow C(\T^d)$, see~\eqref{e1}, the series~\eqref{f2} converges unconditionally in $C(\T^d)$. Thus, applying the operator $\D_\j^Q$ in~\eqref{f2}, we get
\begin{equation}\label{f3}
  |\D_\j^Q(f)(\x)|\le \sum_{\elll\in \Z^d} |\D_\j^Q\(\d_{\j+\elll}(f)\)(\x)|.
\end{equation}
Then, taking $L_p(\T^d,\ell_\t)$-norm on both sides of~\eqref{f3}, we have
\begin{equation}\label{f4}
\begin{split}
    \|2^{r|\j|_1}\D_\j^Q(f)(\x)\|_{L_p(\T^d,\ell_\t)}&\le \sum_{\elll\in \Z^d} \|2^{r|\j|_1}\D_\j^Q\(\d_{\j+\elll}(f)\)\|_{L_p(\T^d,\ell_\t)}\\
&=\sum_{\elll_2^d \in \Z^{d-1}}\sum_{\ell_1 <-1}(\dots)+\sum_{\elll_2^d \in \Z^{d-1}}\sum_{\ell_1 \ge -1}(\dots)=S_1+S_2.
\end{split}
\end{equation}
Consider the sum $S_1$. Denoting
$$
\D_{\j_k^d,k}^Q=\prod_{i=k}^d (Q_{j_i}^i-Q_{j_i-1}^i), \quad k=2,\dots,d,
$$
where $Q_{j_i}^i$ is the univariate operator $Q_{j_i}(\cdot,\vp,\w\vp)$ acting on functions in the variable $x_i$,
we obtain
\begin{equation}\label{f5}
\begin{split}
  S_1&=\sum_{\elll_2^d \in \Z^{d-1}}\sum_{\ell_1 <-1}2^{-r\ell_1} \|2^{r(j_1+\ell_1)+r|\j_2^d|_1}(Q_{j_1}^1-Q_{j_1-1}^1)\D_{\j_2^d,2}^Q(\d_{\j+\elll}(f))\|_{L_p(\T^d,\ell_\t)}\\
  &\le \sum_{b\in \{-1,0\}} \sum_{\elll_2^d \in \Z^{d-1}}\sum_{\ell_1 <-1}2^{-r\ell_1} \|2^{r(j_1+\ell_1)+r|\j_2^d|_1}(Q_{j_1+b}^1-I)\D_{\j_2^d,2}^Q(\d_{\j+\elll}(f))\|_{L_p(\T^d,\ell_\t)},
\end{split}
\end{equation}
where $I$ is the identity operator.
Next, taking into account that
$
Q_j(t,\vp,\w\vp)=\w\vp_j*\vp_j*t
$
for any trigonometric polynomial $t\in \mathcal{T}_{j-1}^1$ and using condition 3) and Lemma~\ref{lber}, we derive
\begin{equation}\label{f6}
\begin{split}
\bigg\|2^{r(j_1+\ell_1)+r|\j_2^d|_1} & (Q_{j_1+b}-I)\D_{\j_2^d,2}^Q(\d_{\j+\elll}(f)) \bigg\|_{L_p(\T^d,\ell_\t)}\\
&\lesssim \bigg\|2^{r(j_1+\ell_1)+r|\j_2^d|_1-sj_1} \frac{\partial^s}{\partial x_1^s}\D_{\j_2^d,2}^Q(\d_{\j+\elll}(f)) \bigg\|_{L_p(\T^d,\ell_\t)}\\
&\lesssim 2^{s\ell_1}\|2^{r(j_1+\ell_1)+r|\j_2^d|_1}\D_{\j_2^d,2}^Q(\d_{\j+\elll}(f))\|_{L_p(\T^d,\ell_\t)}\\
&\lesssim 2^{s\ell_1}\|2^{r|\j|_1}\D_{\j_2^d,2}^Q(\d_{j_1,j_2+\ell_2,\dots,j_d+\ell_d}(f))\|_{L_p(\T^d,\ell_\t)}.
\end{split}
\end{equation}
Then, combining~\eqref{f5} and~\eqref{f6} and taking into account that $\sum_{\ell_1<-1}2^{(s-r)\ell_1}<\infty$, we get
\begin{equation}\label{f7}
\begin{split}
S_1\lesssim \sum_{\elll_2^d \in \Z^{d-1}} \|2^{r|\j|_1}\D_{\j_2^d,2}^Q(\d_{j_1,j_2+\ell_2,\dots,j_d+\ell_d}(f))\|_{L_p(\T^d,\ell_\t)}
\end{split}
\end{equation}

Now, we consider the sum $S_2$. We have
\begin{equation}\label{f8}
\begin{split}
  S_2&=\sum_{\elll_2^d \in \Z^{d-1}}\sum_{\ell_1\ge -1}2^{-r\ell_1} \|2^{r(j_1+\ell_1)+r|\j_2^d|_1}(Q_{j_1}^1-Q_{j_1-1}^1)\D_{\j_2^d,2}^Q\(\d_{\j+\elll}(f)\)\|_{L_p(\T^d,\ell_\t)}\\
  &\le \sum_{b\in \{-1,0\}}\sum_{\elll_2^d \in \Z^{d-1}}\sum_{\ell_1 \ge -1}2^{-r\ell_1} \|2^{r(j_1+\ell_1)+r|\j_2^d|_1}Q_{j_1+b}^1\D_{\j_2^d,2}^Q\(\d_{\j+\elll}(f)\)\|_{L_p(\T^d,\ell_\t)}.
\end{split}
\end{equation}
Denote by $v=(v_j)_{j\in \Z_+}$ the sequence of de la Vall\'ee Poussin kernels given by
$$
v_j(x)=\sum_k \eta(2^{-j}k)e^{{\rm i}kx},\quad \eta(\xi)=\left\{
         \begin{array}{ll}
           1, & \hbox{$|\xi|\le 1$,} \\
           2-|\xi|, & \hbox{$1\le |\xi|\le 2$,} \\
           0, & \hbox{$|\xi|\ge 2$.}
         \end{array}
       \right.
$$
It is not difficult to see that
$$
Q_{j_1+b}^1(g,\vp,\w\vp)(x)=\vp_{j_1+b}*Q_{j_1+b}^1(g,v,\w\vp)(x)\quad\text{for every}\quad g\in C(\T^d).
$$
Using this equality, condition 1), and Lemma~\ref{lf} together with condition 2) and the well-known estimate $|v_j(x)|\lesssim 2^j(1+2^j|x|)^{-2}$, we obtain 
\begin{equation}\label{f9}
  \begin{split}
  \|&2^{r(j_1+\ell_1)+r|\j_2^d|_1}Q_{j_1+b}^1\D_{\j_2^d,2}^Q\(\d_{\j+\elll}(f)\)\|_{L_p(\T^d,\ell_\t)}\\
  &=\Big\|2^{r(j_1+\ell_1)+r|\j_2^d|_1}\vp_{j_1+b}*Q_{j_1+b}^1\(\D_{\j_2^d,2}^Q(\d_{\j+\elll}(f)),v,\w\vp\)\Big\|_{L_p(\T^d,\ell_\t)}\\
  &\lesssim  \Big\|2^{r(j_1+\ell_1)+r|\j_2^d|_1}Q_{j_1+b}^1\(\D_{\j_2^d,2}^Q(\d_{\j+\elll}(f)),v,\w\vp\)\Big\|_{L_p(\T^d,\ell_\t)}\\
     &\lesssim 2^{a\ell_1}\Big\|2^{r(j_1+\ell_1)+r|\j_2^d|_1}\( \mathcal{M}_1( \mathcal{P}_{2^{j_1+\ell_1},a,1}\D_{\j_2^d,2}^Q\(\d_{\j+\elll}(f)\) )^\l   \)^{1/\l}\Big\|_{L_p(\T^d,\ell_\t)},
  \end{split}
\end{equation}
where the parameters $\l$ and $a$ are chosen such that $1/2<\l<\min\{p,\t\}$ ($\l=1$ in the case $\min\{p,\t\}>1$) and $\max\{1/p,1/\t\}<a<r$. Next, applying maximal inequalities~\eqref{m1} and~\eqref{m2}, we get
\begin{equation*}
  \begin{split}
&\Big\|2^{r(j_1+\ell_1)+r|\j_2^d|_1}\( \mathcal{M}_1\( \mathcal{P}_{2^{j_1+\ell_1},a,1}\D_{\j_2^d,2}^Q\(\d_{\j+\elll}(f)\) \)^\l   \)^{1/\l}\Big\|_{L_p(\T^d,\ell_\t)}\\
&=\Big\|2^{\l r(j_1+\ell_1)+\l r|\j_2^d|_1}\mathcal{M}_1\( \mathcal{P}_{2^{j_1+\ell_1},a,1}\D_{\j_2^d,2}^Q\(\d_{\j+\elll}(f)\) \)^\l   \Big\|_{L_{p/\l}\(\T^d,\ell_{\t/\l}\)}^{1/\l}\\
&\lesssim \|2^{r(j_1+\ell_1)+r|\j_2^d|_1}\mathcal{P}_{2^{j_1+\ell_1},a,1}\D_{\j_2^d,2}^Q\(\d_{\j+\elll}(f)\)     \|_{L_p(\T^d,\ell_\t)}\\
&\lesssim \|2^{r(j_1+\ell_1)+r|\j_2^d|_1}\D_{\j_2^d,2}^Q\(\d_{\j+\elll}(f)\)     \|_{L_p(\T^d,\ell_\t)}
      =\|2^{r|\j|_1}\D_{\j_2^d,2}^Q(\d_{j_1,j_2+\ell_2,\dots,j_d+\ell_d}(f))\|_{L_p(\T^d,\ell_\t)}.
  \end{split}
\end{equation*}
This together with~\eqref{f9} and~\eqref{f8} implies that
\begin{equation}\label{f11}
  \begin{split}
     S_2 &\lesssim \sum_{b\in \{-1,0\}}\sum_{\elll_2^d \in \Z^{d-1}}\sum_{\ell_1 \ge -1}2^{-(r-a)\ell_1} \|2^{r|\j|_1}\D_{\j_2^d,2}^Q(\d_{j_1,j_2+\ell_2,\dots,j_d+\ell_d}(f))\|_{L_p(\T^d,\ell_\t)}\\
     &\lesssim \sum_{\elll_2^d \in \Z^{d-1}} \|2^{r|\j|_1}\D_{\j_2^d,2}^Q(\d_{j_1,j_2+\ell_2,\dots,j_d+\ell_d}(f))\|_{L_p(\T^d,\ell_\t)}.
  \end{split}
\end{equation}
Now, combining~\eqref{f4}, \eqref{f7}, and~\eqref{f11}, we derive
\begin{equation*}
  \|2^{r|\j|_1}\D_\j^Q(f)(\x)\|_{L_p(\T^d,\ell_\t)}\lesssim \sum_{\elll_2^d \in \Z^{d-1}} \|2^{r|\j|_1}\D_{\j_2^d,2}^Q\d_{j_1,j_2+\ell_2,\dots,j_d+\ell_d}(f)\|_{L_p(\T^d,\ell_\t)}.
\end{equation*}
Repeating the above procedure for the indices $\ell_2,\dots, \ell_d$, we get the upper estimate in~\eqref{tlpf.1} in the case $r>\max\{1/p,1/\t\}$.

To obtain the upper estimate in~\eqref{tlpf.1} for all $r>0$, we use Lemma~\ref{li}
and follow the idea given in~\cite{S92}. For this, we consider the linear operator $R\,:\, \mathbf{F}_{p,\t}^{r}(\T^d) \mapsto L_p(\T^d,\ell_\t^r)$ defined by $Rf=\big(\D_\j^Q (f)\big)_{\j\in \Z_+^d}$.
In view of the equality $\mathbf{F}_{p,p}^{r}(\T^d)=\mathbf{B}_{p,p}^{r}(\T^d)$ and Theorem~\ref{tb1}, the operator $R$ is bounded if $r>0$ and $p=\t$. Above, we proved that this operator is also bounded for all $1\le p<\infty$ and $1\le \t\le \infty$ if $r>\max\{1/p,1/\t\}$. Hence, applying Lemma~\ref{li}, we get that the operator $R$ is bounded as a mapping with respect to the intermediate spaces  $\mathbf{F}_{p,\t}^{r}(\T^d) \mapsto L_p(\T^d,\ell_\t^r)$ with $r>0$, $1\le p<\infty$, and $1\le \t<\infty$. This proves the upper estimate in~\eqref{tlpf.1} for all cases of parameters.
\end{proof}

\begin{remark}
It follows from the above proof that Theorem~\ref{tf1} is also valid in the case $\t=\infty$ if we additionally suppose that $r>1/p$.
\end{remark}

\begin{remark}
Note that Theorem~\ref{tf1} is not valid in the case $r=0$. Otherwise, using the arguments from the proof of Theorems~5 and~6 in~\cite{SU06}, we would be able to obtain the same rate of convergence for the operators $T_n^Q$ as one given in~\eqref{S2}, but according to the lower estimates established in Theorem~\ref{tb3}, this is impossible.
\end{remark}

\subsection{Error estimates}

\begin{theorem}\label{tf2}
  Let $1\le p,\t<\infty$, $1\le q\le \infty$, $(1/p-1/q)_+<r<s$,  and let $Q=\(Q_j(\cdot,\vp,\w\vp)\)_{j\in \Z_+}$, where
$\vp=\(\vp_j\)_{j\in \Z_+}$ and $\w\vp=\(\w\vp_j\)_{j\in \Z_+}$ satisfy conditions 1)--3) of Theorem~\ref{tf1}.
Then, for every $f\in \mathbf{F}_{p,\t}^r(\T^d)$ and $n\in \N$, we have
\begin{equation}\label{f13}
    \|f-T_n^Q(f)\|_q \lesssim \|f\|_{\mathbf{F}_{p,\t}^r}
\left\{
    \begin{array}{ll}
    2^{-r n}n^{(d-1)(1-\frac1\t)}, & \hbox{$q\le p$,} \\
    2^{-(r-\frac1p+\frac1q) n}, & \hbox{$p<q<\infty$,} \\
    2^{-(r-\frac1p) n}n^{(d-1)(1-\frac1p)}, & \hbox{$p<q=\infty$.}
    \end{array}
\right.
\end{equation}
\end{theorem}

\begin{proof}
The proof is quite standard and similar to the proof of Theorem~\ref{tb2}. For convenience of the reader, we give a proof of this theorem for all cases of parameters.

First, by Theorem~\ref{tf1}, we have
\begin{equation}\label{f14}
  \|f-T_n^Q(f)\|_q=\bigg\|\sum_{|\j|_1>n}\D_\j^Q(f)\bigg\|_q.
\end{equation}

1) The case $q\le p$. Using~\eqref{f14}, twice H\"older's inequality, and estimates~\eqref{sum} and~\eqref{tlpf.1}, we derive
\begin{equation*}
  \begin{split}
     &\|f-T_n^Q(f)\|_q\lesssim \bigg\|\sum_{|\j|_1>n}|\D_\j^Q(f)|\bigg\|_p\\
&\lesssim \left\|\(\sum_{|\j|_1>n} 2^{-r\t'|\j|_1}\)^{1/\t'} \!\!\(\sum_{|\j|_1>n} 2^{r\t|\j|_1}|\D_\j^Q(f)|^\t\)^{1/\t}\right\|_p\lesssim 2^{-r n}n^{(d-1)(1-\frac1\t)} \| f\|_{\mathbf{F}_{p,\t}^r}.
   \end{split}
\end{equation*}

2) The case $p<q<\infty$.  We get from~\eqref{f14} and~\eqref{tlpf.1} that
\begin{equation*}
  \begin{split}
     \|f-T_n^Q(f)\|_q&\lesssim 2^{-(r-\frac1p+\frac1q)n} \bigg\|\sum_{|\j|_1>n}2^{-(r-\frac1p+\frac1q)|\j|_1}|\D_\j^Q(f)|\bigg\|_q\\
&\lesssim 2^{-(r-\frac1p+\frac1q)n} \|f\|_{\mathbf{F}_{q,1}^{r-1/p+1/q}}.
   \end{split}
\end{equation*}
This together with the embedding $\mathbf{F}_{p,\t}^{r}(\T^d)\hookrightarrow\mathbf{F}_{q,1}^{r-1/p+1/q}(\T^d)$, see~\eqref{e4}, implies the required estimate in~\eqref{f13}.

3) The case $p<q=\infty$. Applying~\eqref{f14}, triangle inequality, Nikolskii's and H\"older's inequalities along with estimates~\eqref{sum} and~\eqref{tlpf.1}, we obtain
\begin{equation}\label{f15}
  \begin{split}
     &\|f-T_n^Q(f)\|_\infty\lesssim \sum_{|\j|_1>n}2^{|\j|_1/{\w p}}\|\D_\j^Q(f)\|_{\w p}\\
&\lesssim\(\sum_{|\j|_1>n} 2^{-(r-1/p)p'|\j|_1}\)^{1/p'} \(\sum_{|\j|_1>n} 2^{rp(1-1/p+1/\w p)|\j|_1}\|\D_\j^Q(f)\|_{\w p}^p\)^{1/p}\\
&\lesssim 2^{-(r-\frac1p) n}n^{(d-1)(1-\frac1p)} \| f\|_{\mathbf{B}_{\w p,p}^{r-1/p+1/\w p}}\,.
   \end{split}
\end{equation}
Next, using the Jawerth--Franke-type embedding $\mathbf{F}_{p,\t}^{r}(\T^d) \hookrightarrow \mathbf{B}_{\w p,p}^{r-1/p+1/\w p}(\T^d)$ (see~\cite{HV09} and~\cite[Remark~3.8]{DTU18}), we get that~\eqref{f15} implies~\eqref{f13} in the case $q=\infty$.
\end{proof}

\begin{theorem}\label{tf3}
  Let $1\le p,\t<\infty$, $1\le q\le \infty$, $(1/p-1/q)_+<r<s$, and let $Q=\(Q_j(\cdot,\vp,\w\vp)\)_{j\in \Z_+}$, where
$\vp=\(\vp_j\)_{j}$ and $\w\vp=\(\w\vp_j\)_{j}$ satisfy conditions 1)--3) of Theorem~\ref{tf1}. Suppose additionally that
$\h{\vp_j}(0)=1$, $j\in \Z_+$, and condition~\eqref{cond} holds for $\w\vp_j$.
Then
\begin{equation}\label{t2.1}
   \sup_{f\in U\mathbf{F}_{p,\t}^r }\|f-T_n^Q(f)\|_q \asymp
\left\{
    \begin{array}{ll}
     2^{-r n}n^{(d-1)(1-\frac1\t)}, & \hbox{$q\le p$,} \\
     2^{-(r-\frac1p+\frac1q) n}, & \hbox{$p<q<\infty$,} \\
     2^{-(r-\frac1p) n}n^{(d-1)(1-\frac1p)}, & \hbox{$q=\infty$, $\min\{2,p\}\le\t$.}
    \end{array}
\right.
\end{equation}
\end{theorem}

\begin{proof}
In view of Theorem~\ref{tf2}, we only need to consider the estimates from below.
In the case $q\le p$, the proof of the corresponding inequality is the same as the one given in Theorem~\ref{tb3}. We only note that the function $g_n=c2^{-rn}n^{-(d-1)/\t}f_n$ from the proof of that theorem belongs also to $U\mathbf{F}_{p,\t}^r$. To establish the lower estimate in~\eqref{t2.1} for the cases $p<q<\infty$ and $(p,q)=(1,\infty)$, it suffices to consider the function $f(\x)=c 2^{-n(r+1-1/p)}\Phi_{n+4,0,\dots,0}(\x)$, cf.~\eqref{phij}, and to use~\eqref{bel3}. Finally, considering the case $1<p<q=\infty$, we have that if $2\le \t$, then the proof of~\eqref{t2.1} follows from Lemma~\ref{lbel0} and
the embedding $\mathbf{F}_{2,\t}^r(\T^d) \hookrightarrow \mathbf{F}_{p,\t}^r(\T^d)$, see~\eqref{e3}. If $p\le \t$, then the proof of~\eqref{t2.1} follows from Theorem~\ref{tb3} and the embedding $\mathbf{B}_{p,p}^r(\T^d) \hookrightarrow \mathbf{F}_{p,\t}^r(\T^d)$, see~\eqref{e2}.
\end{proof}

\section{Examples}\label{sec6}

In this section, we give several examples to illustrate Theorems~\ref{tb1}--\ref{tb3} and~\ref{tf1}--\ref{tf3}. Consider the Kantorovich-type operators $K_j^\mathcal{D}(f)=Q_j(f,\vp,\w\vp)$ generated by the family of the normalized characteristic functions $\w\vp_j(x)=2^{j+\s}\chi_{[-\frac{\pi}{2^{j+\s}},\frac{\pi}{2^{j+\s}}]}(x)$, $\s\ge 2$, $j\in \Z_+$,  and the modified Dirichlet kernels $\vp_j(x)=\mathcal{D}_j(x)=\sum_{\ell\in A_j} e^{{\rm i}\ell x}$, $j\in \Z_+$, i.e.,
\begin{equation}\label{ek1}
  K_j^\mathcal{D}(f)(x)=\sum_{k\in A_j} \frac{2^{\s-1}}{\pi}\int_{-\pi2^{-j-\s}}^{\pi2^{-j-\s}} f\big(t+x_k^j\big){\rm d}t\, \mathcal{D}_j\big(x-x_k^j\big).
\end{equation}
Note that for such operators, it suffices to consider only the conditions of Theorems~\ref{tb1} and~\ref{tf1} with $1<p<\infty$.

It is well known that $\sup_{j}\|(\chi_{[-2^{j-1},2^{j-1})}(k))_{k\in \Z}\|_{M_p(\T)}<\infty$ for $1<p<\infty$, see, e.g.,~\cite[2.1.3]{TB}. Also, it is easy to see that $\sup_{j}\|\w\vp_j\|_{\mathcal{L}_{p',j}}<\infty$. Thus, conditions 1) and 2) of Theorem~\ref{tb1} for $1<p<\infty$ are satisfied. To verify condition 3), we consider the function
$$
g(\xi)=\frac{1-\frac{\sin 2^{-\s}\pi \xi}{2^{-\s}\pi \xi}}{\xi^2}\phi_0(4\xi).
$$
It is not difficult to see that $g$ and $g'$ belong to $L_2(\R)$. By Lemma~\ref{lmult-}, this implies that $\sup_j \|(g(2^{-j}k))_{k}\|_{M_p(\T)}<\infty$. Thus, taking into account the equality
$$
g(2^{-j}k)=\frac{1-\h{\vp_j}(k)\h{\w\vp_j}(k)}{(2^{-j}k)^2}\phi_0(2^{-j+2}k),\quad k\in \Z,\,\, j\in \Z_+,
$$
we have that condition 3) of Theorem~\ref{tb1} is also satisfied for the families $\vp$ and $\w\vp$. In particular, the following result holds.

\begin{corollary}\label{corex1}
 Let $1<p<\infty$, $1\le \t\le \infty$ and let $Q=(K_j^\mathcal{D})_{j\in \Z_+}$.
 Then, for every $f\in \mathbf{B}_{p,\t}^r(\T^d)$, $0<r<2$, representation~\eqref{tlpb.0} and equivalence~\eqref{tlpb.1} are valid.
\end{corollary}

Now we consider the conditions of Theorem~\ref{tf1} that partially coincide with those of Theorem~\ref{tb1} verified above. Let $\j\in \Z_+^d$ and $i\in \{1,\dots,d\}$. Denote $B_{\j,i}=\{t\in \T^d\,:\, -2^{j_k+b_{k,i}}\le t_k<2^{j_k+b_{k,i}},\, k=1,\dots,d\}$, where $b_{k,i}=0$ if $k=i$ and $b_{k,i}=2$ otherwise.  According to \cite[3.4.3, Remark~3]{ST87}), we have that $\((\chi_{B_{\j,i}}(\k))_{\k\in \Z^d}\)_{\j\in \Z_+^d} \in M_p(\T^d,\ell_\t)$ for $1<p<\infty$, which implies that $\(\(\h{\vp_{j_i}}(k_i)\)_{\k\in \Z^d}\)_{\j\in \Z_+^d} \in M_p(\T^d,\ell_\t)$. Hence, condition 1) of Theorem~\ref{tf1} is satisfied.
Next, it is easy to see that inequality~\eqref{lf1+} obviously holds and by the same arguments as above, considering
the functions
$$
h_i(\bm{\xi})=\frac{1-\frac{\sin 2^{-\s}\pi \xi_i}{2^{-\s}\pi \xi_i}}{\xi_i^2}\phi_0(4\xi_i)\prod_{\nu=1}^d \phi_0(\xi_\nu),\quad i=1,\dots,d,
$$
and using Lemma~\ref{lmult}, we derive that the corresponding families $(\mu^{(2)}_{\j,i}(\vp,\w\vp))_{\j\in \Z_+^d}$ belong to  $M_p(\T^d,\ell_\t)$ for each $i=1,\dots,d$. Thus, all conditions of Theorem~\ref{tf1}  are satisfied and we obtain the following result.

\begin{corollary}\label{corex2}
 Let $1<p,\t<\infty$ and let $Q=(Q_j(\cdot,\vp,\w\vp))_{j\in \Z_+}$, where $Q_j(\cdot,\vp,\w\vp)$ is defined in~\eqref{ek1}
 Then, for every $f\in \mathbf{F}_{p,\t}^r(\T^d)$, $0<r<2$, representation~\eqref{tlpf.0} and equivalence~\eqref{tlpf.1} are valid.
\end{corollary}

Note that, dealing with operators~\eqref{ek1}, we can apply Theorems~\ref{tb1}--\ref{tb3} and~\ref{tf1}--\ref{tf3} only for functions
$f$ from the spaces $\mathbf{B}_{p,\t}^r(\T^d)$ and $\mathbf{F}_{p,\t}^r(\T^d)$ with $r\in (0,2)$. To consider functions of higher smoothness, we need to modify a little the operators defined in~\eqref{ek1}. One possibility is to replace the Dirichlet kernel $\mathcal{D}_j$ by the kernel $\vp_j(x)=\mathcal{D}_j^*(x)=\sum_{\ell\in A_j}\frac{\pi2^{-j-\s}\ell}{\sin \pi 2^{-j-\s}\ell} {\rm e}^{{\rm i} \ell x}$. In this case, we obviously have that $1-\h{\vp_j}(k)\h{\w\vp_j}(k)=0$ for $k\in A_j$, and the corresponding compatibility conditions hold for all $s>0$. Another possibility, is to replace the Dirichlet kernel $\mathcal{D}_j$ by its linear combinations $\mathcal{D}_{j,s}^*(x)=\sum_{\nu=0}^N a_{\nu} \mathcal{D}_j(x-\frac{\pi \nu}{2^{j+\s}})$, where $s$ is a given parameter and $N$ and $\a_{\nu}$ are chosen such that
$$
1-\frac{\sin \xi}{\xi}\sum_{\nu=0}^N  a_\nu {\rm e}^{i\nu\xi}=\a\xi^s+\mathcal{O}(\xi^{s+1}),\quad \xi\to 0.
$$
Then, using Lemmas~\ref{lmult-} and~\ref{lmult} and the argument above, one can verify that the corresponding compatibility conditions in Theorems~\ref{tb1}--\ref{tb3} and~\ref{tf1}--\ref{tf3} hold for a given $s>0$. For example, for $s=3$, we can take $N=2$ and
$a_0=\frac56$, $a_1=\frac26$, $a_2=-\frac16$. Similarly, instead of linear combinations of the Dirichlet kernel one can consider linear combinations of the corresponding characteristic function $\w\vp_j$, see, e.g.,~\cite{KS17}.

It is also worth noting that we can take de la Vall\'ee Poussin kernels as well as Fej\'er or Riesz kernels instead of the Dirichlet kernel $\mathcal{D}_j$ in formula~\eqref{ek1}. In this case, using Lemmas~\ref{lmult-} and~\ref{lmult}, the above arguments can be easily extended to parameter $p=1$ and $\infty$. In particular, for the operators
\begin{equation}\label{ex3+}
  K_j^\mathcal{V}(f)(x)=\sum_{k\in A_j} \frac{2^{\s-1}}{\pi}\int_{-\pi 2^{-j-\s}}^{\pi2^{-j-\s}} f\big(t+x_k^j\big){\rm d}t\, v_j\big(x-x_k^j\big)
\end{equation}
generated by the classical de la Vall\'ee Poussin kernel $v_j(x)=\sum_{\ell} \eta(2^{-j}\ell)e^{{\rm i}\ell x}$ with
$$
\eta(\xi)=\left\{
            \begin{array}{ll}
              1, & \hbox{$|\xi|\le 1/2$,} \\
              2(1-|\xi|), & \hbox{$1/2\le |\xi|\le 1$,} \\
              0, & \hbox{$|\xi|\ge 1$,}
            \end{array}
          \right.
$$
we have the following corollaries.

\begin{corollary}\label{corex3}
 Let $1\le p, \t\le \infty$ and let $Q=(K_j^\mathcal{V})_{j\in \Z_+}$.
 Then, for every $f\in \mathbf{B}_{p,\t}^r(\T^d)$, $0<r<2$, representation~\eqref{tlpb.0} and equivalence~\eqref{tlpb.1} are valid.
\end{corollary}

\begin{corollary}\label{corex4}
 Let $1\le p,\t<\infty$ and let $Q=(K_j^\mathcal{V})_{j\in \Z_+}$.
 Then, for every $f\in \mathbf{F}_{p,\t}^r(\T^d)$, $0<r<2$, representation~\eqref{tlpf.0} and equivalence~\eqref{tlpf.1} are valid.
\end{corollary}

\section{Remarks about approximation by sampling-type operators}\label{sec7}
In the above theorems, we studied approximation properties of the Smolyak algorithm generated by families of quasi-interpolation operators $Q=\(Q_j(\cdot,\vp,\w\vp)\)_{j\in \Z_+}$, where $\w\vp=\(\w\vp_j\)_{j}$ is a sequence of functions satisfying $\sup_j \|\w\vp_j\|_{\mathcal{L}_{p',j}(\T)}<\infty$. Note that instead of sequences of functions $\w\vp=\(\w\vp_j\)_{j}$, we can deal with sequences of distributions, which for example generate sampling-type operators on the Smolyak grid.
In particular, taking into account Remark~\ref{rem1}, it is not difficult to verify that Theorems~\ref{tb1} and~\ref{tb2} remain valid if we replace condition 2) by condition~\eqref{dis} with $\s\ge 0$.  The same is valid for Theorems~\ref{tf1} and~\ref{tf2} if we take $\s=0$ in~\eqref{dis}. A typical example of a distribution $\w\vp_j$, which can be considered in the mentioned theorems, is a linear combination of the shifted delta functions, i.e.,
$$
\w\vp_j(x)=\sum_{|\nu|\le N_j}a_{\nu}^j \d(x-t_\nu^j),\quad a_\nu\in \C,\quad t_\nu^j\in \T.
$$
In this case, the corresponding sampling operators are defined by
$$
S_j(f,\vp)(x)=Q_j(f,\vp,\w\vp)(x)=2^{-j}\sum_{k\in A_j} \(\sum_{|\nu|\le N_j} a_{\nu}^j f(x_k^j-t_\nu^j)\)\vp_j(x-x_k^j).
$$
For these operators, we have the following analogues of Theorems~\ref{tb1} and~\ref{tf1}.

\begin{proposition}\label{prb1}
Let $1\le p, \t\le \infty$, $1/p<r<s$, and let $Q=\(S_j(\cdot,\vp)\)_{j\in \Z_+}$, where
the sequence
$\vp=\(\vp_j\)_{j}$  satisfies the following conditions:
\begin{itemize}
  \item[$1)$] $\vp_j\in \mathcal{T}_j^1$ and $\sup_j \|(\h{\vp_j}(k))_k\|_{M_p(\T)}<\infty$;
  \item[$2)$] there exists $\d>0$ such that
  \begin{equation*}
  \sup_j\bigg\|\bigg(\frac{1-\h{\vp_j}(k)\sum_{|\nu|\le N_j} a_{\nu}^j e^{-{\rm i}kt_\nu^j}}{(2^{-j}{\rm i}k)^s}\phi_0(\d 2^{-j}k)\bigg)_{\!\!k}\bigg\|_{M_p(\T)}<\infty.
\end{equation*}
\end{itemize}
Then, for every $f\in \mathbf{B}_{p,\t}^r(\T^d)$, representation~\eqref{tlpb.0} and equivalence~\eqref{tlpb.1} are valid.
\end{proposition}

\begin{proposition}\label{prf1}
  Let $1\le p,\t<\infty$, $\max\{1/p,1/\t\}<r<s$,  and let $Q=\(S_j(\cdot,\vp)\)_{j\in \Z_+}$, where
the sequence
$\vp=\(\vp_j\)_{j\in \Z_+}$  satisfies conditions of Proposition~\ref{prb1} and additionally, for each $i=1,\dots,d$,
$$
\(\(\h{\vp_{j_i}}(k_i)\)_{\k}\)_{\j\in \Z_+^d} \in M_p(\T^d,\ell_\t),
$$
$$
\(\bigg(\frac{1-\h{\vp_{j_i}}(k_i)\sum_{|\nu|\le N_j} a_{\nu}^j e^{-{\rm i}k_it_\nu^j}}{(2^{-j_i} {\rm i}k_i)^s}\bigg)_{\k\in \Z^d}\)_{\j\in \Z_+^d}\in M_p(\T^d,\ell_\t).
$$
Then, for every $f\in \mathbf{F}_{p,\t}^r(\T^d)$, representation~\eqref{tlpf.0} and equivalence~\eqref{tlpf.1} are valid.
\end{proposition}

\bigskip

\noindent\textbf{Acknowledgements} 
 The author is indebted to the referees for thorough reading and valuable suggestions and comments.

                         \end{document}